\documentclass[12pt]{amsart}
\usepackage{amsmath,amssymb,amsfonts,amsthm,amsopn}
\usepackage{graphics}
\usepackage{xcolor}
\usepackage{pb-diagram}
\usepackage[title]{appendix}

\newcommand{\tf}{time-frequency}

\newcommand{\tfs}{time-frequency shift}

\newtheorem{theorem}{Theorem}[section]
\newtheorem{lemma}[theorem]{Lemma}
\newtheorem{corollary}[theorem]{Corollary}
\newtheorem{proposition}[theorem]{Proposition}
\newtheorem{definition}[theorem]{Definition}

\newtheorem{example}[theorem]{Example}
\newtheorem{remark}[theorem]{Remark}

\newcommand{\beqa}{\begin{eqnarray*}}
\newcommand{\eeqa}{\end{eqnarray*}}

\DeclareMathOperator*{\Lift}{Lift}

\newcommand{\field}[1]{\mathbb{#1}}
\newcommand{\bR}{\field{R}}        
\newcommand{\bC}{\field{C}}        
\def\la{\lambda}

\def\cF{\mathcal{F}}              
\def\cS{\mathcal{S}}
\def\cD{\mathcal{D}}

\def\cE{\mathcal{E}}
\def\cG{\mathcal{G}}
\def\cM{\mathcal{M}}
\def\cK{\mathcal{K}}

\def\cA{\mathcal{A}}

\def\cI{\mathcal{I}}

\def\rd{\bR^d}

\def\rdd{{\bR^{2d}}}

\def\lrd{L^2(\rd)}

\def\R{\right)}

\def\<{\left<}
\def\>{\right>}

\def\mv1{M_v^1}

\def\mpq{M^{p,q}}

\def\phas{(x,\xi )}
\def\mn{(m,n)}
\def\mn'{(m',n')}

\newcommand{\norm}[1]{\lVert#1\rVert}

\def\R{\mathbb{R}}
\def\Ren{\mathbb{R}^d}

\def\sch{\mathcal{S}}

\def\Sn2{S_{2}(L^{2}(\Ren))}
\def\S1{S_{1}(L^{2}(\Ren))}
\def\sig00{\sigma_{0,0}}

\def\la{\langle}
\def\ra{\rangle}

\begin{document}
\begin{abstract} 
We introduce new frames, called \emph{metaplectic Gabor frames}, as natural generalizations of Gabor frames in the framework of metaplectic Wigner distributions, cf. \cite{CR2021, CR2022, CGR2022, Giacchi,Zhang21bis, Zhang21}. Namely, we develop the theory of metaplectic atoms in a full-general setting and prove an inversion formula for metaplectic Wigner distributions on $\rd$. Its discretization provides  metaplectic Gabor frames. 

Next, we deepen the understanding of the so-called shift-invertible metaplectic Wigner distributions, showing that they can be represented, up to chirps,  as rescaled short-time Fourier transforms. 
As an application,  we derive a new characterization of modulation and Wiener amalgam spaces. Thus, these metaplectic distributions (and related frames) provide meaningful definitions of local frequencies and can be used to measure effectively the local frequency content of signals. 
\end{abstract}

\title[Metaplectic Gabor Frames and Symplectic Analysis of spaces]{Metaplectic Gabor Frames and Symplectic Analysis of time-frequency spaces}

\author{Elena Cordero}
\address{Universit\`a di Torino, Dipartimento di Matematica, via Carlo Alberto 10, 10123 Torino, Italy}
\email{elena.cordero@unito.it}
\author{Gianluca Giacchi}
\address{Università di Bologna, Dipartimento di Matematica,  Piazza di Porta San Donato 5, 40126 Bologna, Italy; Institute of Systems Engineering, School of Engineering, HES-SO Valais-Wallis, Rue de l'Industrie 21, 1950 Sion, Switzerland; Lausanne University Hospital and University of Lausanne, Lausanne, Department of Diagnostic and Interventional Radiology, Rue du Bugnon 46, Lausanne 1011, Switzerland. The Sense Innovation and Research Center, Avenue de Provence 82
1007, Lausanne and Ch. de l’Agasse 5, 1950 Sion, Switzerland.}
\email{gianluca.giacchi2@unibo.it}

\thanks{}
\subjclass[2010]{42C15,42B35,42A38}

\keywords{Frames, time-frequency analysis, modulation spaces, Wiener amalgam spaces, time-frequency representations, metaplectic group, symplectic group}
\maketitle

\section{Introduction}
%
Frames were originally  introduced by  Duffin and Schaeffer in \cite{DS1952} and today they have become popular in many different fields, such as sampling theory, phase retrival, operator theory (they almost diagonalize several classes of  pseudodifferential and Fourier integral operators), PDE's, nonlinear
sparse approximation,  wireless communications,  signal processing,
quantum  mechanics and computing (cf. \cite{Elena-book,GrochenigMystery,heil,PILIPOVIC-Stoeva} and references therein). Any environment may require a suitable frame, tailored for the matter, so that the main concern is to find \emph{the right atoms} to represent a function.   

For a  fixed window $g\in L^2(\rd)\setminus\{0\}$, define 
 $M_\xi g(t):=e^{2\pi i\xi\cdot t}g(t)$, $\xi\in\rd$,  and $T_xg(t)=g(t-x)$, $x\in \rd$,  the \textit{modulation}  and the \textit{translation}  operator, respectively.
Their composition $\pi(z)=M_\xi T_x$, $z=\phas$, is called a \emph{time-frequency shift}. Let $\Lambda$ be a sequence of points in $\rdd$ (that is, a \emph{discrete} set in $\rdd$). 
Then the Gabor system generated by $g$ and $\Lambda$ is the set of time-frequency shifts 
$$\cG(g,\Lambda)=\{\pi(\lambda)g\}_{\lambda\in\Lambda}.$$

The Gabor system is a Gabor frame if it is a frame: namely, if  there exist $A,B>0$ such that the inequalities
\begin{equation}\label{defGFintro}
	A\norm{f}_2^2\leq\sum_{\lambda\in\Lambda}|\la f,\pi(\lambda)g\ra|^2\leq B\norm{f}_2^2
\end{equation}
hold for every function $f\in L^2(\rd)$. This implies the \emph{reproducing formula}
\begin{equation}\label{discreteinvSTFTintro}
	f = \sum_{\lambda\in \Lambda}\la f,\pi(\lambda)g\ra\pi(\lambda)\gamma,
\end{equation}
for a suitable function $\gamma\in\lrd$ (called \textit{dual window}).

Observe that the elements of a Gabor frame are simply generated by
time-frequency shifts of a single window function and are called  \emph{Gabor atoms}.
They arise naturally from the discretization of the
 \textit{short-time Fourier transform} (STFT),  defined as 
\begin{equation}\label{STFTintro}
	V_gf(x,\xi)=\la f,M_\xi T_x g\ra, \qquad \phas\in\rdd.
\end{equation}

In fact, the STFT  decomposes a signal $f\in\lrd$ as integral superposition of the time-frequency shifts  $\pi(x,\xi)$ as follows:
\begin{equation}\label{invSTFTintro}
	f = \frac{1}{\la \gamma,g\ra}\int_{\rdd}V_gf(x,\xi)\pi(x,\xi) \gamma dxd\xi, \qquad f\in L^2(\rd),
\end{equation}
where $g,\gamma\in L^2(\rd)$, $\gamma,g\in L^2(\rd)$ satisfy $\la \gamma,g\ra\neq0$, and the integral is intended in the weak sense of vector-valued integration. 

In the practice, integrals are approximated by the partial sums of their Riemann sums, so that, using the equality $|V_gf(\lambda)|=|\la f,\pi(\lambda)g\ra|$,   the Gabor reproducing formula in \eqref{discreteinvSTFTintro} can be viewed as a discretization  of \eqref{invSTFTintro}.  Equivalently, it expresses $f$ as a discrete superposition of fundamental \textit{atoms}. \\

In this paper we are mainly concerned with the characterization of the fundamental  spaces in time-frequency analysis: modulation and Wiener amalgam spaces.
They were introduced by H. Feichtinger in the 80's \cite{Feichtinger_1981_Banach} (cf. Galperin and Samarah \cite{Galperin2004} for the quasi-Banach setting) and have become popular in the last twenty years, since they have been proved to be the right environment for many different topics:  signal analysis, PDE's,  quantum mechanics, approximation theory \cite{KB2020,Elena-book}.

 Let $m$ be a $v$-moderate weight, see Section \ref{subsec:Lpq} below for details. We say that a tempered distribution $f$ belongs to the \textit{modulation space} $M^{p,q}_m(\rd)$, $0<p,q\leq\infty$, if $V_gf\in L^{p,q}_m(\rdd)$. Consequently, these spaces are used to measure the local time-frequency content of signals in terms of Lebesgue (quasi-)norms.

Differently from the framework of $L^p$ spaces, where the convolution is not even well-defined for $L^p$ functions with $0<p<1$, discrete convolution inequalities hold also in the quasi-Banach setting. These facts, together with Gabor theory, are  used  to prove the atomic decomposition of modulation spaces \cite{book,Galperin2004}. 
Namely, if $\cG(g,\Lambda)$ is a Gabor frame, formula (\ref{discreteinvSTFTintro}) holds with unconditional convergence in $M^{p,q}_m(\rd)$ ($0<p,q<\infty$) and with weak-$\ast$ convergence in $M^\infty_{1/v}(\rd)$ otherwise. Moreover,
\begin{align*}\label{atomicDecintro}
	\norm{f}_{M^{p,q}_m}&\asymp \norm{(V_gf(\lambda_1,\lambda_2))_{(\lambda_1,\lambda_2)\in\Lambda}}_{\ell^{p,q}_m(\Lambda)}\\
	&=\left(\sum_{\lambda_2}\left(\sum_{\lambda_1}|V_gf(\lambda_1,\lambda_2)|^p
m(\lambda_1,\lambda_2)^p\right)^{q/p}\right)^{1/q}.\notag
\end{align*}


In this paper we extend the characterization above to more general frames, including the Gabor ones. As well as Gabor frames arise as discretization of the reproducing formula for the STFT, we introduce frames that come from discretizations of a more general class of TF-representations, including the STFT. Namely, the latter representation, as well as the  \textit{$\tau$-Wigner distributions} (see Section \ref{subsec:23} below), are examples of the so-called \textit{metaplectic Wigner distributions}, introduced in \cite{CR2021} and studied in \cite{CR2022, CGR2022, Giacchi}.

 For a fixed metaplectic operator $\hat\cA\in Mp(2d,\bR)$, the metaplectic Wigner distribution $W_\cA$ is defined by
\begin{equation}\label{defMetWigintro}
	W_\cA(f,g)=\hat\cA(f\otimes \bar g),\quad f,g\in L^2(\rd).
\end{equation}
We refer to Section \ref{subsec:26} for the definition of metaplectic operators. If the pointwise evaluations $W_\cA(f,g)(x,\xi)$, $x,\xi\in\rdd$, are well defined, $W_\cA(f,g)(x,\cdot)$ can also be used to represent the local frequency content of the signal $f$ at time $x$ differently and more suitably, according to the context. For this reason, it is important to establish whether a metaplectic Wigner distribution can  be used to measure the local frequency content of signals or, stated differently, when 
\begin{equation}\label{charBanachintro}
	\norm{f}_{M^{p,q}_m}\asymp\norm{W_\cA(f,g)}_{L^{p,q}_m}.
\end{equation}

For any metaplectic Wigner distribution $W_\cA$ there exists a matrix $E_\cA\in \bR^{2d\times 2d}$ such that
\[
	|W_\cA(\pi(w)f,g)(z)|=|W_\cA(f,g)(z-E_\cA w)|, \qquad w\in\rdd,
\]
and $W_\cA$ is \textit{shift-invertible} if $E_\cA\in GL(2d,\bR)$. It was shown in \cite{CGshiftinvertible} that if $W_\cA$ is shift-invertible and $E_\cA$ is upper-triangular, then (\ref{charBanachintro}) holds for all $1\leq p,q\leq\infty$. Nevertheless, the nature of shift-invertible Wigner distributions was still poorly-understood and no explicit characterization of them was provided. \\

In the attempt to prove (\ref{charBanachintro}) for the quasi-Banach setting $0<p,q\leq\infty$, the question arises whether an equivalent of (\ref{invSTFTintro}) can be proved for metaplectic Wigner distributions. Roughly speaking, Moyal's identity:
\[
	\la W_\cA(f,g),W_\cA(\varphi,\gamma)\ra = \la f,\varphi\ra \overline{\la g,\gamma\ra},\quad f,g,\varphi,\gamma\in L^2(\rd),
\]
 implies that
\begin{equation}\label{metAtomsintro}
	\la f,\varphi\ra = \frac{1}{\la \gamma,g\ra}\int_{\rdd}W_\cA(f,g)(z)\overline{W_\cA(\varphi,\gamma)(z)}dz.
\end{equation}
This suggests to define the \textit{metaplectic atoms} $\pi_\cA(z)$, $z\in\rd$, implicitly on $\cS(\rd)$ as the distributions characterized by:
\[
\la \varphi,\pi_\cA(z)\gamma\ra={ W_\cA(\varphi,\gamma)(z) }, \qquad \varphi\in\cS(\rd),
\]
so that (\ref{metAtomsintro}) becomes the vector-valued integral:
\[
	f = \frac{1}{\la\gamma,g\ra}\int_{\rdd}W_\cA(f,g)(z)\pi_\cA(z)\gamma dz.
\]
A \textit{metaplectic Gabor system} of $L^2(\rd)$ is defined  as the family $$\cG_\cA(g,\Lambda)=\{\pi_\cA(\lambda)g\}_{\lambda\in\Lambda},$$ with $g\in L^2(\rd)$ and $\Lambda\subset\rdd$ a discrete set. If the family above is a frame, that is, it satisfies the inequalities in \eqref{defGFintro} with $\pi_\cA(\lambda)g$ in place of $\pi(\lambda)g$, we call it a \emph{metaplectic Gabor frame}. \\

In this work, we develop the theory of metaplectic Gabor frames, showing that the related frame operator enjoys similar property to the Gabor one.  In particular, in Theorem \ref{thmFrames} below, under the shift-invertibility assumption of $W_\cA$ it is shown the equivalence of the following statements:

\vspace{0.1truecm}
\noindent
\emph{(i) $\cG_\cA(g,\Lambda)$ is a metaplectic Gabor frame with bounds $A,B$;\\
(ii) $\cG(\widehat{\delta_\cA}g,E_\cA^{-1}\Lambda)$ is a Gabor frame with bounds $|\det(E_\cA)|A,|\det(E_\cA)|B$;}\\

\vspace{0.1truecm}\noindent
with $\widehat{\delta_\cA}$ being a suitable metaplectic operator called \emph{deformation operator}, see Definition 4.5 in the sequel.

Special instances of metaplectic Gabor frames are the $\hbar$-Gabor frames introduced by M. de Gosson in \cite{DGosson}, see Example \ref{exDG2} in Section 6.
This result generalizes \cite[Proposition 7]{DGosson} because in our case $E_\cA$ needs not to be symplectic.\\

Another outcome of this manuscript is given by the characterization of the shift-invertibility property of $W_\cA$.
We prove that $W_\cA$ is shift-invertible if and only if $W_\cA$ is, roughly speaking,  a STFT up to linear change of variables and products-by-chirps (Corollary \ref{cor43} in Section \ref{sec:SIU}):

\vspace{0.1truecm}
\emph{ $W_\cA$ is shift-invertible if and only if, up to a sign, for any $f,g\in L^2(\rd)$} 
$$
	W_\cA(f,g)(z)=|\det(E_\cA)|^{-1/2}\Phi_{N_\cA}(E_\cA^{-1}z)V_{\widehat{\delta_\cA} g}f(E_\cA^{-1}z),\quad z\in\rdd,
$$
where $	\Phi_{N_\cA}(t)=e^{\pi i t\cdot N_\cA t},\,\,t\in\rd,$
with an appropriate matrix $N_\cA$.

This characterization shows that the property of measuring local time-frequency content of signals is basically a typical feature of the STFT. \\
As application,  we complete the characterization of modulation and Wiener amalgam spaces started in \cite{CR2022, CGR2022, CGshiftinvertible}, cf. Theorem \ref{thmF} below. This result shows that, under the shift-invertibility assumption, the characterization in \eqref{charBanachintro} holds for every $0<p,q\leq\infty$.\\

{\bf Outline.}  This work is divided as follows. We present preliminaries and notation in Section \ref{sec:preliminaries}. Section \ref{sec:TFatoms} is devoted to metaplectic atoms, defined implicitly as in (\ref{metAtomsintro}), and to an equivalent of inversion formula (\ref{invSTFTintro}) for metaplectic Wigner distributions. In Section \ref{sec:SIU}, we characterize shift-invertible Wigner distributions in terms of the STFT. 
We compute the metaplectic atoms of the distributions which belong to the Cohen's class in Section \ref{sec:CA}. In Section \ref{sec:MGF} we define metaplectic Gabor frames, characterizing those related to shift-invertible distributions. In Section \ref{sec:CMS} we complete the characterization of modulation spaces and Wiener amalgams in terms of shift-invertibility. We devote the Appendix to the proof of an intertwining formula between metaplectic operators and complex conjugation, which is used to obtain the expression of the adjoint of metaplectic atoms in Section \ref{sec:TFatoms}.

\section{Preliminaries}\label{sec:preliminaries}
\textbf{Notation.} We denote 
$xy=x\cdot y$ (scalar product on $\Ren$).  The space   $\sch(\Ren)$ is the Schwartz class, which is a Frech\'{e}t space with seminorms
\[
	\rho_{\alpha,\beta}(f):=\sup_{x\in\rd}|x^\alpha D^\beta f(x)|, \qquad \alpha,\beta\in\mathbb{N}^d,
\]
whereas its dual $\sch'(\Ren)$ is the space of tempered distributions. The brackets  $\la f,g\ra$ denote the extension to $\sch' (\Ren)\times\sch (\Ren)$ of the inner product $\la f,g\ra=\int f(t){\overline {g(t)}}dt$ on $L^2(\Ren)$ (conjugate-linear in the second component). We write a point in the phase space (or \tf\ space) as
$z=(x,\xi)\in\rdd$, and  the corresponding phase-space shift (\tfs )
acts on a function or distribution  as
\begin{equation}
\label{eq:kh25}
\pi (z)f(t) = e^{2\pi i \xi\cdot t} f(t-x), \, \quad t\in\rd.
\end{equation}
In the following, we will use the composition law of time-frequency shifts: for all $z=(z_1,z_2),w=(w_1,w_2)\in\rdd$,
\begin{equation}\label{commLawTFs}
	\pi(z)\pi(w)=e^{-2\pi i z_1\cdot w_2}\pi(z+w).
\end{equation}
Trivially $\pi(0)=id_{L^2}$ and it is easy to verify that
\begin{equation}
\label{inverseTFs}
	\pi(z)^{-1}=\pi(z)^\ast=e^{-2\pi iz_1\cdot z_2}\pi(-z).
\end{equation}
Time-frequency shifts are isometries of $L^2(\rd)$. If $t\in\rd$, the Dirac delta distribution $\delta_t\in\cS'(\rd)$ is characterized by
\[
	\la\delta_t,\varphi\ra:=\overline{\varphi(t)} \qquad \varphi\in\cS(\rd).
\]

The notation $f\lesssim g$ means that there exists $C>0$ such that $ f(x)\leq Cg(x)$ holds for all $x$. The symbol $\lesssim_t$ is used when we stress that $C=C(t)$. If $ g\lesssim f\lesssim g$ or, equivalently, $ f \lesssim g\lesssim f$, we write $f\asymp g$. For two measurable functions $f,g:\rd\to\bC$, we set $f\otimes g(x,y):=f(x)g(y)$. If $X,Y$ are vector spaces, $X\otimes Y$ is the unique completion of $\text{span}\{x\otimes y : x\in X, y\in Y\}$. If $X(\rd)=L^2(\rd)$ or $\cS(\rd)$, the set $\text{span}\{f\otimes g:f,g\in X(\rd)\}$ is dense in $X(\rdd)$. Thus, for all $f,g\in\cS'(\rd)$, the operator $f\otimes g\in\cS'(\rdd)$ is defined by its action on $\varphi\otimes\psi\in\cS(\rdd)$ by
\[
	\la f\otimes g,\varphi\otimes\psi\ra = \la f,\varphi\ra\la g,\psi\ra
\]
extends uniquely to a tempered distribution of $\cS'(\rdd)$. 

$GL(d,\bR)$ denotes the group of $d\times d$ invertible matrices.


\subsection{Weighted mixed norm spaces}\label{subsec:Lpq}
We denote by $v$ a continuous, positive, even, submultiplicative weight function on $\rdd$, i.e., 
$ v(z_1+z_2)\leq v(z_1)v(z_2)$, for all $ z_1,z_2\in\rdd$. 
We say that $w\in \mathcal{M}_v(\rdd)$ if $w$ is a positive, continuous, even weight function on $\rdd$ that is {\it
	$v$-moderate}:
$ w(z_1+z_2)\lesssim v(z_1)w(z_2)$  for all $z_1,z_2\in\rdd$. Fundamental examples are the polynomial weights
\begin{equation}\label{vs}
	v_s(z) =(1+|z|)^{s},\quad s\in\bR,\quad z\in\rdd.
\end{equation}
Two weights $m_1,m_2$ are equivalent if $m_1\asymp m_2$. For example, $v_s(z)\asymp (1+|z|^2)^{s/2}$.\\

If $m\in\cM_v(\rdd)$, $0<p,q\leq\infty$ and $f:\rdd\to\bC$ measurable, we set 
\[
	\norm{f}_{L^{p,q}_m}:=\left(\int_{\rd}\left(\int_{\rd}|f(x,y)|^pm(x,y)^p dx\right)^{q/p}dy\right)^{1/q}=\norm{y\mapsto\norm{f(\cdot,y)m(\cdot,y)}_{p}}_{q},
\]
with the obvious adjustments when $\min\{p,q\}=\infty$. The space of measurable functions $f$ having $\norm{f}_{L^{p,q}_m}<\infty$ is denoted by $L^{p,q}_m(\rdd)$. 
We recall the following partial generalization of the results in \cite{Fuhr}, which gathers the content of \cite[Theorems A2 and A3]{CGshiftinvertible}:

\begin{proposition}\label{thmA12}
		(i) Consider $A,D\in GL(d,\bR)$, $B\in \bR^{d\times d}$ and $0<p,q\leq\infty$. Define the {\bf upper triangular} matrix
		\begin{equation}\label{uppertr}
			S=\begin{pmatrix}
				A & B\\
				0_{d\times d} & D
			\end{pmatrix}.
		\end{equation}
		The mapping $\mathfrak{T}_S:f\in L^{p,q}(\rdd)\to |\det(S)|^{1/2}f\circ S$  is an isomorphism of $L^{p,q}(\rdd)$ with bounded inverse $\mathfrak{T}_{S^{-1}}$.\\
		(ii) Let $m\in\mathcal{M}_v(\rdd)$, $S\in GL(2d,\bR)$ and $0<p,q\leq\infty$. Consider the operator $(\mathfrak{T}_S)_m: f\in L^{p,q}_m(\rdd)\mapsto |\det(S)|^{1/2}f\circ S.$ If $m\circ S\asymp m$, then
	$\mathfrak{T}_S:L^{p,q}(\rdd)\to L^{p,q}(\rdd)$ is bounded if and only if $(\mathfrak{T}_S)_m:L^{p,q}_m(\rdd)\to L^{p,q}_m(\rdd)$ is bounded.
	\end{proposition}

\subsection{Fourier transform}
In this work, the Fourier transform of $f\in \cS(\rd)$ is defined as
\[
	\hat f(\xi)=\int_{\rd} f(x)e^{-2\pi i\xi\cdot x}dx, \qquad \xi\in\rd.
\]
If $f\in\cS'(\rd)$, the Fourier transform of $f$ is defined by duality as the tempered distribution characterized by
\[
	\langle \hat f,\hat\varphi\rangle=\langle f,\varphi\rangle, \qquad \varphi\in\cS(\rd).
\]
We denote with $\cF f:=\hat f$ the Fourier transform operator. It is a surjective automorphism of $\cS(\rd)$ and $\cS'(\rd)$, as well as a surjective isometry of $L^2(\rd)$.

If $f\in\cS(\rdd)$, we set $\cF_2f(x,\eta):=\int_{\rd}f(x,y)e^{-2\pi i\eta\cdot y}dy$, the partial Fourier transform with respect to the second variables, which is a surjective isomorphism of $\cS(\rdd)$ to itself. This definition extends to $L^2(\rdd)$ by density and to $\cS'(\rdd)$ by duality. Namely, if $f\in\cS'(\rdd)$, $$ \langle \cF_2f,\varphi\rangle = \langle f,\cF_2^{-1}\varphi\rangle, \qquad \varphi\in\cS(\rdd).$$
%

\subsection{Time-frequency analysis tools}\label{subsec:23}
The \textit{short-time Fourier transform} of $f\in L^2(\rd)$ with respect to the window $g\in L^2(\rd)$ is the time-frequency representation defined as
\[
	V_gf(x,\xi)=\int_{\rd}f(t)\overline{g(t-x)}e^{-2\pi i\xi\cdot t}dt, \qquad (x,\xi)\in\rdd.
\]
This definition extends to $(f,g)\in\cS'(\rd)\times\cS(\rd)$ by antilinear duality as $V_gf(x,\xi)=\langle f,\pi(x,\xi)g\rangle$. 
The reproducing formula for the STFT reads as follows: for all $g,\gamma\in L^2(\rd)$ such that $\la g,\gamma\ra\neq0$,
\begin{equation}\label{STFTinv}
	f=\frac{1}{\la\gamma,g\ra}\int_{\rdd}V_gf(x,\xi)\pi(x,\xi)\gamma  dxd\xi,
\end{equation}
where the identity holds in $L^2(\rd)$ as a vector-valued integral in the weak sense (see, e.g., \cite[Subsection 1.2.4]{Elena-book}).

In high-dimensional complex features information processing $\tau$-Wigner distributions  ($\tau\in\bR$) play a crucial role \cite{ZJQ21}. They are defined as
\begin{equation}\label{tauWigner}
	W_\tau(f,g)(x,\xi)=\int_{\rd} f(x+\tau t)\overline{g(x-(1-\tau)t)}e^{-2\pi i\xi \cdot t}dt, \qquad  \phas\in\rd,
\end{equation}
for $f,g\in L^2(\rd)$.
The cases $\tau=0$ and $\tau=1$ are the so-called (cross-)\textit{Rihacek distribution}
\begin{equation}\label{RD}
W_0(f,g)(x,\xi)=f(x)\overline{\hat g(\xi)}e^{-2\pi i\xi\cdot x}, \quad \phas\in\rd,
\end{equation}
 and (cross-)\textit{conjugate Rihacek distribution}
 \begin{equation}\label{CRD}
 W_1(f,g)(x,\xi)=\hat f(\xi)\overline{g(x)}e^{2\pi i\xi\cdot x}, \quad \phas\in\rd.
 \end{equation}

\subsection{Modulation  spaces \cite{KB2020,F1,Feichtinger_1981_Banach,book,Galperin2004,Kobayashi2006,PILIPOVIC2004194}} \label{subsec:MSs}
Fix $0<p,q\leq\infty$, $m\in\mathcal{M}_v(\rdd)$, and $g\in\cS(\rd)\setminus\{0\}$. The \textit{modulation space} $M^{p,q}_m(\rd)$ is classically defined as the space of tempered distributions $f\in\cS'(\rd)$ such that $$\norm{f}_{M^{p,q}_m}:=\Vert V_gf\Vert_{L^{p,q}_m}<\infty.$$ If $\min\{p,q\}\geq1$, the quantity $\norm{\cdot}_{M^{p,q}_m}$ defines a norm, otherwise a quasi-norm. Different windows give rise to equivalent (quasi-)norms. Modulation spaces are (quasi-)Banach spaces and the following continuous inclusions hold: \\
if $0<p_1\leq p_2\leq\infty$, $0<q_1\leq q_2\leq\infty$ and $m_1,m_2\in\mathcal{M}_{v}(\rdd)$ satisfy $m_2\lesssim m_1$: $$ \cS(\rd)\hookrightarrow M^{p_1,q_1}_{m_1}(\rd)\hookrightarrow M^{p_2,q_2}_{m_2}(\rd)\hookrightarrow\cS'(\rd).$$ In particular, $M^1_v(\rd)\hookrightarrow M^{p,q}_m(\rd)$ whenever $m\in\mathcal{M}_v(\rdd)$ and $\min\{p,q\}\geq1$. 
We denote with $\cM^{p,q}_m(\rd)$ the closure of $\cS(\rd)$ in $M^{p,q}_m(\rd)$, which coincides with the latter whenever $p,q<\infty$. Moreover, if $1\leq p,q<\infty$, $(M^{p,q}_m(\rd))'=M^{p',q'}_{1/m}(\rd)$, where $p'$ and $q'$ denote the Lebesgue conjugate exponents of $p$ and $q$ respectively. Finally, if $m_1\asymp m_2$, then $M^{p,q}_{m_1}(\rd)=M^{p,q}_{m_2}(\rd)$ for all $p,q$.
%
%

\subsection{The symplectic group $Sp(d,\mathbb{R})$ and the metaplectic operators}\label{subsec:26}
	A matrix $S\in\bR^{2d\times 2d}$ is symplectic, we write $S\in Sp(d,\bR)$, if 
	\begin{equation}\label{fundIdSymp}
	S^TJS=J,\end{equation} where the matrix $J$ is defined as
	\begin{equation}\label{defJ}
	J=\begin{pmatrix}
		0_{d\times d} & I_{d\times d}\\
		-I_{d\times d} & 0_{d\times d}
	\end{pmatrix}.
\end{equation}
	In this work, $I_{d\times d}\in\bR^{d\times d}$ is the identity matrix and $0_{d\times d}$ is the matrix of $\bR^{d\times d}$ having all zero entries.
	
	We represent $S\in Sp(d,\bR)$ as a block matrix 
	\begin{equation}\label{blocksA}
		S=\begin{pmatrix} A & B\\
		C & D\end{pmatrix}
	\end{equation}
	with $A,B,C,D\in\bR^{d\times d}$. It is straightforward to verify that $S\in\bR^{2d\times 2d}$ is symplectic if and only if the following conditions hold:
	\begin{align*}
			&(R1) \qquad \text{$A^TC=C^TA$},\\
			&(R2) \qquad \text{$B^TD=D^TB$},\\
			&(R3) \qquad \text{$A^TD-C^TB=I_{d\times d}$},
	\end{align*}
	and it can be proved that $\det(S)=1$ and the inverse of $S$ is explicitly given in terms of the blocks of $S$ as
	\begin{equation}\label{invSymp}
		S^{-1}=\begin{pmatrix} D^T & -B^T\\
		-C^T & A^T\end{pmatrix}.
	\end{equation}
  
	For $E\in GL(d,\bR)$ and $C\in\bR^{d\times d}$, $C$ symmetric, we define
	\begin{equation}\label{defDLVC}
		\cD_E:=\begin{pmatrix}
			E^{-1} & 0_{d\times d}\\
			0_{d\times d} & E^T
		\end{pmatrix} \qquad \text{and} \qquad V_C:=\begin{pmatrix}
			I_{d\times d} & 0\\ C & I_{d\times d}
		\end{pmatrix}.
	\end{equation}
	$J$ and the matrices in the form $V_C$ ($C$ symmetric) and $\cD_E$ ($E$ invertible) generate the group $Sp(d,\bR)$.\\ 

Let $\rho$ be the Schr\"odinger representation of the Heisenberg group, that is $$\rho(x,\xi;\tau)=e^{2\pi i\tau}e^{-\pi i\xi\cdot x}\pi(x,\xi),$$ for all $x,\xi\in\rd$, $\tau\in\bR$. We will use the following tensor product property:  for all $f,g\in L^2(\rd)$, $z=(z_1,z_2),w=(w_1,w_2)\in\rdd$,
\[
	\rho(z;\tau)f\otimes\rho(w;\tau)g=e^{2\pi i\tau}\rho(z_1,w_1,z_2,w_2;\tau)(f\otimes g).
\]
For all $S\in Sp(d,\bR)$, $\rho_S(x,\xi;\tau):=\rho(S (x,\xi);\tau)$ defines another representation of the Heisenberg group that is equivalent to $\rho$, i.e., there exists a unitary operator $\hat S:L^2(\rd)\to L^2(\rd)$ such that
\begin{equation}\label{muAdef}
	\hat S\rho(x,\xi;\tau)\hat S^{-1}=\rho(S(x,\xi);\tau), \qquad  x,\xi\in\rd, \ \tau\in\bR.
\end{equation}
This operator is not unique, but if $\hat S'$ is another unitary operator satisfying (\ref{muAdef}), then $\hat S'=c\hat S$, for some constant $c\in\bC$, $|c|=1$. The set $\{\hat S : S\in Sp(d,\bR)\}$ is a group under composition and it admits a subgroup that contains exactly two operators for each $S\in Sp(d,\bR)$. This subgroup is called \textbf{metaplectic group}, denoted by $Mp(d,\bR)$. It is a realization of the two-fold cover of $Sp(d,\bR)$ and the projection \begin{equation}\label{piMp}
	\pi^{Mp}:Mp(d,\bR)\to Sp(d,\bR)
\end{equation} is a group homomorphism with kernel $\ker(\pi^{Mp})=\{-id_{{L^2}},id_{{L^2}}\}$.

Throughout this work, if $\hat S\in Mp(d,\bR)$, the matrix $S$ (without the caret) will always be the unique symplectic matrix such that $\pi^{Mp}(\hat S)=S$.

\begin{proposition}{\cite[Proposition 4.27]{folland89}}\label{Folland427}
	Every operator $\hat S\in Mp(d,\bR)$ maps $\cS(\rd)$ isomorphically to $\cS(\rd)$ and it extends to an isomorphism on $\cS'(\rd)$.
\end{proposition}

For $C\in\R^{d\times d}$, define 
\begin{equation}\label{chirp}
	\Phi_C(t)=e^{\pi i t\cdot Ct},\quad t\in\rd.
\end{equation}
If we add the assumptions  $C$ symmetric and invertible, then we can compute explicitly its Fourier transform, that is  
\begin{equation}\label{ft-chirp}
\widehat{\Phi_C}=|\det(C)|\,\Phi_{-C^{-1}}.
\end{equation}

\begin{example}\label{es22} For certain $\hat S\in Mp(d,\bR)$, the projection $S$ is known. Let $J$, $\cD_L$ and $V_C$ be defined as in (\ref{defJ}) and (\ref{defDLVC}), respectively. Then,
	\begin{enumerate}
		\item[\it (i)] $\pi^{Mp}(\cF)=J$;
		\item[\it (ii)] if $\mathfrak{T}_E:=|\det(E)|^{1/2}\,f(E\cdot)$, then $\pi^{Mp}(\mathfrak{T}_E)=\cD_E$;
		\item[\it (iii)] if $\phi_C f=\Phi_C f$, then $\pi^{Mp}(\phi_C)=V_C$;
		\item[\it (iv)] if $\psi_C =\cF \Phi_{-C} \cF^{-1}$, then $\pi^{Mp}(\psi_C)f=V_C^T$;
		\item[\it (v)] if $\cF_2 $ is the Fourier transform with respect to the second variables, then $\pi^{Mp}(\cF_2)=\cA_{FT2}$, where $\cA_{FT2}\in Sp(2d,\bR)$ is the $4d\times4d$ matrix with block decomposition
		\begin{equation}\label{AFT2}
		\cA_{FT2}:=\begin{pmatrix}
			I_{d\times d} & 0_{d\times d} & 0_{d\times d} & 0_{d\times d}\\
			0_{d\times d} & 0_{d\times d} & 0_{d\times d} & I_{d\times d} \\
			0_{d\times d} & 0_{d\times d} & I_{d\times d} & 0_{d\times d}\\
			0_{d\times d} & -I_{d\times d} & 0_{d\times d} & 0_{d\times d}
		\end{pmatrix}.
		\end{equation}
	\end{enumerate}

\end{example}

\subsection{Metaplectic Wigner distribution}
Let $\hat\cA\in Mp(2d,\bR)$. The \textbf{metaplectic Wigner distributions} associated to $\hat\cA$ is defined for all $f,g\in L^2(\rd)$ as
\[
	W_\cA(f,g)=\hat\cA(f\otimes\bar g).
\]
All the time-frequency representations of Section \ref{subsec:23} are metaplectic Wigner distributions. Namely, $V_gf=\hat A_{ST}(f\otimes\bar g)$ and $W_\tau(f,g)=\hat A_\tau(f\otimes\bar g)$, where:
\begin{equation}\label{AST}
	A_{ST}=\begin{pmatrix}
		I_{d\times d} & -I_{d\times d} & 0_{d\times d} & 0_{d\times d}\\
		0_{d\times d} & 0_{d\times d} & I_{d\times d} & I_{d\times d}\\
		0_{d\times d} & 0_{d\times d} & 0_{d\times d} & -I_{d\times d}\\
		-I_{d\times d} & 0_{d\times d} & 0_{d\times d} &0_{d\times d}
	\end{pmatrix}
\end{equation}
and
\begin{equation}\label{Atau}
	A_\tau=\begin{pmatrix}
		(1-\tau)I_{d\times d} & \tau I_{d\times d} & 0_{d\times d} & 0_{d\times d}\\
		0_{d\times d} & 0_{d\times d} & \tau I_{d\times d} & -(1-\tau)I_{d\times d}\\
		0_{d\times d} & 0_{d\times d} & I_{d\times d} & I_{d\times d}\\
		-I_{d\times d} & I_{d\times d} & 0_{d\times d} & 0_{d\times d}
	\end{pmatrix}.
\end{equation}
We recall the following continuity properties.
\begin{proposition}\label{prop25}
	Let $W_\cA$ be a metaplectic Wigner distribution. Then,\\
	(i) $W_\cA:L^2(\rd)\times L^2(\rd)\to L^2(\rdd)$ is bounded;\\
	(ii) $W_\cA:\cS(\rd)\times\cS(\rd)\to \cS(\rdd)$ is bounded;\\
	(iii) $W_\cA:\cS'(\rd)\times\cS'(\rd)\to\cS'(\rdd)$ is bounded.
\end{proposition}
Moreover, since metaplectic operators are unitary, for all $f_1,f_2,g_1,g_2\in L^2(\rd)$,
\begin{equation}\label{Moyal}
	\la W_\cA(f_1,f_2),W_\cA(g_1,g_2)\ra = \la f_1,g_1\ra \overline{\la f_2,g_2\ra}.
\end{equation}

The projection of a metaplectic operator $\hat\cA\in Mp(2d,\bR)$ is a symplectic matrix $\cA\in Sp(2d,\bR)$ with block decomposition
\begin{equation}\label{blockDecA}
	\cA=\begin{pmatrix}
		A_{11} & A_{12} & A_{13} & A_{14}\\
		A_{21} & A_{22} & A_{23} & A_{24}\\
		A_{31} & A_{32} & A_{33} & A_{34}\\
		A_{41} & A_{42} & A_{43} & A_{44}
	\end{pmatrix}.
\end{equation}
For a $4d\times4d$ symplectic matrix with block decomposition (\ref{blockDecA}), relations $(R1)$ - $(R3)$ read as:
\begin{align*}
&\begin{cases}
	(R1a) & \text{$A_{11}^T A_{31}+A_{21}^T A_{41}=A_{31}^T A_{11}+A_{41}^T A_{21}$},\\
	(R1b) &  \text{$A_{11}^T A_{32}+A_{21}^T A_{42}=A_{31}^T A_{12}+A_{41}^T A_{22}$},\\
	(R1c) &  \text{$A_{12}^T A_{32}+A_{22}^T A_{42}=A_{32}^T A_{12}+A_{42}^T A_{22}$},
	\end{cases}\\
	&\begin{cases}
	(R2a) &  \text{$A_{13}^T A_{33}+A_{23}^T A_{43}=A_{33}^T A_{13}+A_{43}^T A_{23}$},\\
	(R2b) & \text{$A_{13}^T  A_{34}+A_{23}^T A_{44}=A_{33}^T A_{14}+A_{43}^T A_{24}$},\\
	(R2c) & \text{$A_{14}^T A_{34}+A_{24}^T A_{44}=A_{34}^T A_{14}+A_{44}^T A_{24}$},
	\end{cases}\\
	&
	\begin{cases}
	(R3a) &  \text{$A_{11}^T A_{33}+A_{21}^T A_{43}-(A_{31}^T A_{13}+A_{41}^T A_{23})=I_{d\times d}$},\\
	(R3b) &  \text{$A_{11}^T A_{34}+A_{21}^T A_{44}=A_{31}^T A_{14}+A_{41}^T A_{24}$},\\
	(R3c) &  \text{$A_{12}^T A_{33}+A_{22}^T A_{43}=A_{32}^T A_{13}+A_{42}^T A_{23}$},\\
	(R3d) &  \text{$A_{12}^T A_{34}+A_{22}^T A_{44}-(A_{32}^T A_{14}+A_{42}^TA_{24} )=I_{d\times d}$}.
	\end{cases}
\end{align*}
We identify four $2d\times2d$ submatrices of $4d\times4d$ symplectic matrices. Namely, if $\cA\in Sp(2d,\bR)$ has block decomposition (\ref{blockDecA}), we set:
\begin{equation}\label{defEAFA}
	E_\cA=\begin{pmatrix}
		A_{11} & A_{13}\\
		A_{21} & A_{23}
	\end{pmatrix}, \quad F_\cA=\begin{pmatrix}
		A_{31} & A_{33}\\
		A_{41} & A_{43}
	\end{pmatrix},
\end{equation}
and
\begin{equation}\label{defeafa}
		\cE_\cA=\begin{pmatrix}
		A_{12} & A_{14}\\
		A_{22} & A_{24}
	\end{pmatrix}, \quad \cF_\cA=\begin{pmatrix}
		A_{32} & A_{34}\\
		A_{42} & A_{44}
	\end{pmatrix}.
\end{equation}
A simple comparison shows that relationships $(R1a)-(R3d)$ read, in terms of these four submatrices, as
\begin{equation}\label{relInTermsOfSubm}
	\begin{cases}
		E_\cA^TF_\cA-F_\cA^TE_\cA=J,\\
		\cE_\cA^T\cF_\cA-\cF_\cA^T\cE_\cA=J,\\
		E_\cA^T\cF_\cA-F_\cA^T\cE_\cA=0_{d\times d}.
	\end{cases}
\end{equation}
We will also consider
\begin{equation}\label{defBA}
	B_\cA=\begin{pmatrix}
			A_{13} & \frac{1}{2}I_{d\times d}-A_{11}\\
			\frac{1}{2}I_{d\times d}-A_{11}^T & -A_{21}
\end{pmatrix}.
\end{equation}
Finally, the following matrices will appear ubiquitously throughout this work: 
\begin{equation}\label{defL}
	L=\begin{pmatrix}
		0_{d\times d} & I_{d\times d}\\
		I_{d\times d} & 0_{d\times d}
	\end{pmatrix} \quad and \quad P=\begin{pmatrix}
		0_{d	\times d} & I_{d\times d}\\
		0_{d\times d} & 0_{d\times d}
	\end{pmatrix}.
\end{equation}

\begin{lemma}\label{lemma44}
	Let $\cA\in Sp(2d,\bR)$ have block decomposition (\ref{blockDecA}) and $E_\cA,F_\cA,\cE_\cA,\cF_\cA$ be defined as in (\ref{defEAFA}) and (\ref{defeafa}). Let $L$ be defined as in (\ref{defL}). \\
	If $E_\cA\in GL(2d,\bR)$, then, \\
	(i) $\cF_\cA=E_\cA^{-T}F_\cA^T\cE_\cA$;\\
	(ii) the matrix $G_\cA:=LE_\cA^{-1}\cE_\cA$ is symplectic;\\
	(iii) $\cE_\cA\in GL(2d,\bR)$ and $\det(\cE_\cA)=(-1)^d\det(E_\cA)$.\\
	If $\cE_\cA\in GL(2d,\bR)$, then,\\
	(iv) $F_\cA = \cE_\cA^{-T}\cF_\cA^T E_\cA$;\\
	(v) the matrix $\mathfrak{G}_\cA=L\cE_\cA^{-1}E_\cA$ is symplectic;\\
	(vi) $E_\cA\in GL(2d,\bR)$ and $\det(E_\cA)=(-1)^d\det(\cE_\cA)$.\\
	In particular, $E_\cA$ is invertible if and only if $\cE_\cA$ is invertible.
\end{lemma}
\begin{proof}
Relation 	$(i)$ follows directly from the third equation in (\ref{relInTermsOfSubm}), using the invertibility of $E_\cA$.\\
	Item $(ii)$ is a consequence of (\ref{relInTermsOfSubm}) and $(i)$. For, observe that $LJL=-J$, so that:
	\begin{align*}
		G_\cA^TJG_\cA&=(LE_\cA^{-1}\cE_\cA)^TJ(LE_\cA^{-1}\cE_\cA)=\cE_\cA^T E_\cA^{-T} (LJL) E_\cA^{-1}\cE_\cA\\
		&=-\cE_\cA^T E_\cA^{-T} J E_\cA^{-1}\cE_\cA=\cE_\cA^T E_\cA^{-T} (F_\cA^T E_\cA-E_\cA^T F_\cA) E_\cA^{-1}\cE_\cA\\
		&=\cE_\cA^T(E_\cA^{-T}F_\cA^T-F_\cA E_\cA^{-1})\cE_\cA=\cE_\cA^T(E_\cA^{-T}F_\cA^T \cE_\cA) -(\cE_\cA^T F_\cA E_\cA^{-1})\cE_\cA\\
		&=\cE_\cA^T\cF_\cA-\cF_\cA^T\cE_\cA=J.
	\end{align*}
	Finally, $(iii)$ follows directly from $(ii)$. Items $(iv)$-$(vi)$ are proved analogously.
	\end{proof}


\section{Metaplectic atoms}\label{sec:TFatoms}
We start by generalizing the definition of time-frequency shifts. Differently from the classical theory, where time-frequency shifts are defined in terms of translations and modulations, and then used to define the STFT, we define them implicitly from metaplectic Wigner distributions. 

\begin{definition}\label{def1}
	Let $W_\cA$ be a metaplectic Wigner distribution and $z\in\rdd$. The \textbf{metaplectic atom} $\pi_\cA(z)$ is the operator defined by its action on all $f\in\cS(\rd)$ as
	\[
		\langle\varphi,\pi_\cA(z)f\ra := W_\cA(\varphi,f)(z),\qquad \varphi\in\cS(\rd).
	\]
\end{definition}

Observe that if $f,\varphi\in\cS(\rd)$, $W_\cA(\varphi,f)(z)$ is well-defined for all $z\in\rdd$, by Proposition \ref{prop25}.\\

\begin{remark}\label{rem32}
Definition \ref{def1} says that metaplectic atoms play the game of time-frequency shifts for the STFT. 
\end{remark}

Metaplectic atoms map $\cS(\rd)$ to $\cS'(\rd)$, see Proposition \ref{wellposdefMetAt} below. We put this detail aside and take it for granted in favour of some prior example. 

\begin{example}
	The metaplectic atoms associated to the STFT are the time-frequency shifts. In fact, for all $f,\varphi\in\cS(\rd)$ and all $z\in\rdd$
	\[
		\la\varphi,\pi_{A_{ST}}(z)f\ra={V_f \varphi(z)}={\langle \varphi,\pi(z)f \rangle}.
	\]
	This implies that $\pi_{A_{ST}}(z)f$ and $\pi(z)f$ are tempered distributions with the same action on $\cS(\rd)$, i.e. $\pi_{A_{ST}}(z)f=\pi(z)f$.
\end{example}
\begin{example}\label{exDG}
	For $\hbar>0$, consider the time-frequency representation defined for all $f,g\in L^2(\rd)$ by
	\[
		V^\hbar_gf(x,\xi)=\la f,(2\pi \hbar)^{-d/2}\pi^\hbar(x,\xi)g\ra, \qquad (x,\xi)\in\rdd,
	\]
	where $\pi^\hbar(x,\xi)g(t):=e^{i(\xi t-x\cdot \xi/2)/\hbar}g(t-x)$. These are essentially the time-frequency representations considered by M. de Gosson in \cite{DGosson}. For all $\hbar>0$, up to a sign, 
	 {\[
		V_g^\hbar f(x,\xi)=(2\pi \hbar)^{-d/2}e^{2\pi i\frac{x\cdot \xi}{4\pi\hbar}}V_gf\left(x,\frac{\xi}{2\pi\hbar}\right),\qquad \phas\in\rd, \quad f,g\in L^2(\rd),
	\]
	so that} $V^\hbar_gf=W_{\cA_\hbar}(f,g)$, where
	\begin{equation}\label{Ahbar}
		\cA_\hbar=\begin{pmatrix}
			I_{d\times d} & -I_{d\times d} & 0_{d\times d} & 0_{d\times d}\\
			0_{d\times d} & 0_{d\times d} & 2\pi \hbar I_{d\times d} & 2\pi \hbar I_{d\times d}\\
			0_{d\times d} & 0_{d\times d} &  {\frac{1}{2}I_{d\times d}} & {-\frac{1}{2}I_{d\times d}}\\
			-\frac{1}{4\pi\hbar}I_{d\times d} & -\frac{1}{4\pi\hbar}I_{d\times d} & 0_{d\times d} & 0_{d\times d}
		\end{pmatrix}.
	\end{equation}
	In this case, we observe that
	\begin{equation}\label{EAhbar}
		E_{\cA_\hbar}=\begin{pmatrix}
		I_{d\times d} & 0_{d\times d}\\
		0_{d\times d} & 2\pi \hbar I_{d\times d}
		\end{pmatrix}.
	\end{equation}
	The metaplectic atoms associated to $V^\hbar$ are 
	\[
		 {\pi_{\cA_\hbar}(x,\xi)g=(2\pi \hbar)^{-d/2}e^{-i\frac{x\cdot \xi}{2\hbar}}\pi\left(x,\frac{\xi}{2\pi\hbar}\right)g}=|\det(E_{\cA_\hbar})|^{-1/2}e^{-i\frac{x\cdot\xi}{\hbar}}\pi(E_{\cA_\hbar}^{-1}(x,\xi))g,
		 \] 
		 $(x,\xi)\in\rdd$, $g\in \cS(\rd)$.
\end{example}
\begin{example}\label{example34}
	We compute the metaplectic atoms associated to the $\tau$-Wigner distribution $W_\tau$ ($0<\tau<1$). For, let $z=(x,\xi)\in\rdd$ and $f,\varphi\in\cS(\rd)$. Then,
	\[\begin{split}
		{W_\tau(\varphi,f)(x,\xi)}&={\int_{\rd}\varphi(x+\tau t)\overline{f(x-(1-\tau)t)}e^{-2\pi i\xi\cdot t}dt}\\
		&=\frac{1}{\tau^d}\int_{\rd}{\varphi(s)}\overline{f\left(x-(1-\tau)\left(\frac{s-x}{\tau}\right)\right)}e^{-2\pi i\xi\cdot(\frac{s-x}{\tau})}ds\\
		&=\int_{\rd}{\varphi(s)}\overline{f\left(\frac{1}{\tau}x-\frac{1-\tau}{\tau}s\right)}e^{-2\pi i\xi\cdot\frac{s}{\tau}}e^{2\pi i\xi\cdot\frac{x}{\tau}}\frac{ds}{\tau^d}\\
		&=\langle\varphi, \pi_{A_{\tau}}(x,\xi)f\rangle,
	\end{split}\]
	where, if $\mathfrak{T}_\tau f(t)=\frac{(1-\tau)^{d/2}}{\tau^{d/2}}f\left(-\frac{1-\tau}{\tau}t\right)$,
	\[\begin{split}
		\pi_{A_{\tau}}(x,\xi)f(t)&=\frac{1}{\tau^d}e^{-2\pi i\frac{\xi\cdot x}{\tau}}e^{2\pi it\cdot\frac{\xi}{\tau}}f\left(\left(-\frac{1-\tau}{\tau}\right)\left(t-\frac{1}{1-\tau}x\right)\right)\\
		&=\frac{1}{\tau^{d/2}(1-\tau)^{d/2}}e^{-2\pi i\frac{x\cdot\xi}{\tau}}M_{\frac{\xi}{\tau}}T_{\frac{x}{1-\tau}}\mathfrak{T}_\tau f(t).
	\end{split}\]
	Observe that $\frac{1}{\tau^{d/2}|\tau-1|^{d/2}}=|\det(E_{A_\tau})|^{-1/2}$, so
	\[\begin{split}
		\pi_{A_{\tau}}(x,\xi)f&=|\det(E_{A_\tau})|^{-1/2}e^{-2\pi i\frac{x\cdot\xi}{\tau}}\pi\Big(\frac{1}{1-\tau}x,\frac{1}{\tau}\xi\Big)\mathfrak{T}_\tau f\\
		&=|\det(E_{A_\tau})|^{-1/2}e^{-2\pi i\frac{x\cdot\xi}{\tau}}\pi(E_{A_{\tau}}^{-1}(x,\xi))\mathfrak{T}_\tau f.
	\end{split}
	\]
\end{example}

\begin{example}\label{ex32}
	Consider the (cross)-Rihacek distribution $W_0$, defined for all $f,g\in L^2(\rd)$ as
	\[
		W_0(f,g)(x,\xi)=f(x)\overline{\hat g(\xi)}e^{-2\pi i\xi\cdot x},\quad\phas \in\rdd.
	\]
	Then, if $z=(x,\xi)\in \rdd$, $f,g\in\cS(\rd)$,
\begin{equation}\label{rihac}\begin{split}
	\la\varphi, \pi_{A_0}(z)f\ra = {\varphi(x)}\overline{\hat f(\xi)}e^{-2\pi i\xi\cdot x}=\la \varphi, \hat f(\xi)e^{2\pi i\xi \cdot x}T_x\delta_0 \ra.
\end{split}
\end{equation}
Observe that $\pi_{A_0}(x,\xi)f=\hat f(\xi)e^{2\pi i\xi \cdot x}T_x\delta_0$ is a tempered distribution that does not define a function. 
\end{example}

\begin{example}\label{ex4}
	Let $\hat S\in Sp(d,\bR)$ with $S=\pi^{Mp}(\hat S)$ having block decomposition
	\begin{equation}\label{blockS}
		S=\begin{pmatrix}
			A & B\\ C & D
		\end{pmatrix}
	\end{equation}
	and consider the metaplectic Wigner distribution defined in \cite[Example 4.1 (ii)]{CGshiftinvertible} as
	\[
			\widetilde{\mathcal{U}}_gf(z)=V_g(\hat S f)(z)=W_{\cA}(f, g)(z)=\la f,\hat S^{-1}\pi(z)g\ra=\la f,\pi(S^{-1}z)\hat S^{-1}g\ra, 
	\]
	$f,g\in L^2(\rd)$, $x,\xi\in\rd$, where 
	\[
		\cA=\begin{pmatrix}
			A & -I_{d\times d} & B & 0_{d\times d}\\
			C & 0_{d\times d} & D & I_{d\times d}\\
			0_{d\times d} & 0_{d\times d} & 0_{d\times d} & -I_{d\times d}\\
			-A & 0_{d\times d} & -B & 0_{d\times d}
		\end{pmatrix}.
	\]
	Clearly, $E_\cA=S$ and $\pi_\cA(z)g=\pi(S^{-1}z)\hat S^{-1}g$ for all $z\in\rdd$. 
\end{example}

		As aforementioned, in the previous examples we took on trust that metaplectic atoms map $\cS(\rd)$ to $\cS'(\rd)$. This technicality, along with the linearity of metaplecitc atoms, is proved in the proposition that follows. Nevertheless, Example \ref{ex32} shows that in general $\pi_\cA(z)f$, $f\in\cS(\rd)$, is a tempered distribution that is not induced by any locally integrable function. 
		
\begin{proposition}\label{wellposdefMetAt}
	Let $W_\cA$ be a metaplectic Wigner distribution. For all $z\in\rdd$, $\pi_\cA(z)$ is a well-defined linear operator that maps $\cS(\rd)$ to $\cS'(\rd)$.
\end{proposition}
\begin{proof}
	Let $f\in\cS(\rd)$. By definition, for any $\varphi,\psi\in\cS(\rd)$ and $\alpha\in\bC$,
	\begin{align*}
		\langle\alpha\varphi+\psi,\pi_\cA(z)f\ra &= {W_\cA(\alpha\varphi+\psi,f)(z)}={\hat\cA((\alpha\varphi+\psi)\otimes\bar f)(z)}\\
		&={\hat\cA(\alpha\varphi\otimes \bar f+\psi\otimes\bar f)(z)}={\alpha\hat\cA(\varphi\otimes \bar f)(z)}+{\hat\cA(\psi\otimes\bar f)(z)}\\
		&=\alpha {W_\cA(\varphi,f)(z)}+{W_\cA(\psi,f)(z)}=\alpha\la\varphi,\pi_\cA(z)f\ra+\la \psi,\pi_\cA(z)f\ra.
	\end{align*}
Then, we need to prove that $\pi_\cA(z)f:\varphi\in\cS(\rd)\mapsto\la\varphi,\pi_\cA(z)f\ra\in\bC$ is continuous. Using the boundedness of $W_\cA:\cS(\rd)\times\cS(\rd)\to\cS(\rdd)$, 
\begin{align*}
	|\la\varphi,\pi_\cA(z)f\ra|&=|W_\cA(\varphi,f)(z)|\leq\norm{W_\cA(\varphi,f)}_{L^\infty(\rdd)}=\rho_{0,0}(W_\cA(\varphi,f))\\
	&\leq C \sum_{j=1}^N\rho_{\alpha_j,\beta_j}(\varphi)\sum_{j=1}^M\rho_{\gamma_j\delta_j}(f)=\tilde C\sum_{j=1}^N\rho_{\alpha_j,\beta_j}(\varphi).
\end{align*}
Thus, it remains to check the linearity of $\pi_\cA(z)$. For, let $\alpha\in\bC$, $f,g\in\cS(\rd)$. For every $\varphi\in\cS(\rd)$,
\begin{align*}
	\langle\varphi,\pi_\cA(z)(\alpha f+g)\ra &={W_{\cA}(\varphi,\alpha f+g)(z)}={\hat\cA(\varphi\otimes(\overline{\alpha f+g}))(z)}\\
	&=\bar\alpha{\hat\cA(\varphi\otimes \bar f)(z)}+{\hat\cA(\varphi\otimes\bar g)(z)}\\
	&=\bar\alpha{W_\cA(\varphi,f)(z)}+{W_\cA(\varphi,g)(z)}\\
	&=\bar\alpha\la\varphi,\pi_\cA(z)f\ra+\la\varphi,\pi_\cA(z)g\ra\\
	&=\langle\varphi,\alpha\pi_\cA(z)f+\pi_\cA(z)g\ra.
\end{align*}
This concludes the proof.
\end{proof}

The first question that we address is the validity of an equivalent of the inversion formula (\ref{STFTinv}) for metaplectic Wigner distributions.

\begin{theorem}\label{inversionFormula}
	Let $W_\cA$ be a metaplectic Wigner distribution and $f,g\in L^2(\rd)$. If $\gamma\in \cS(\rd)$ satisfies $\la \gamma,g\ra\neq0$, then
	\begin{equation}
	\label{invFormu}
		f=\frac{1}{\la\gamma,g\ra}\int_{\rdd}W_{\cA}(f,g)(z)\pi_\cA(z)\gamma dz
	\end{equation}
	where the integral must be interpreted in the weak sense of vector-valued integration.
\end{theorem}
\begin{proof}
	We use the definition of vector-valued integral. For $\varphi\in L^2(\rd)$, using (\ref{Moyal}),
	\begin{align*}
		\la\frac{1}{\la \gamma,g\ra}&\int_{\rdd} W_\cA(f,g)(z)\pi_\cA(z)\gamma dz,\varphi\ra=\frac{1}{\la\gamma,g\ra}\int_{\rdd}W_\cA(f,g)(z)\la \pi_\cA(z)\gamma,\varphi\ra dz\\
		&=\frac{1}{\la\gamma,g\ra}\int_{\rdd}W_\cA(f,g)(z) \overline{W_\cA(\varphi,\gamma)(z)}dz=\frac{1}{\la\gamma,g\ra}\la W_\cA(f,g),W_\cA(\varphi,\gamma)\ra\\
		&=\frac{1}{\la\gamma,g\ra} \la f,\varphi\ra \overline{\la g,\gamma\ra}=\la f,\varphi\ra.
	\end{align*}
	This shows (\ref{invFormu}).
\end{proof}


In what follows, we use the definitions of the submatrices $E_\cA$, $F_\cA$, $\cE_\cA$ and $\cF_\cA$ given in (\ref{defEAFA}) and (\ref{defeafa}).

\begin{lemma}\label{commpiWA}
	Let $W_\cA$ be a metaplectic Wigner distribution. Then, for $z\in\rdd$, $f,g\in L^2(\rd)$, we have
	\[
	W_\cA(\pi(z)f,g)=\Phi_{-M_\cA}(z)\pi(E_\cA z,F_\cA z)W_\cA(f,g),
	\]
	where, if $\cA=\pi^{Mp}(\hat\cA)$ has block decomposition (\ref{blockDecA}), $M_\cA$ is the symmetric matrix
	\begin{equation}\label{defMA}
		M_\cA=\begin{pmatrix}
			A_{11}^TA_{31}+A_{21}^TA_{41} & A_{31}^TA_{13}+A_{41}^TA_{23}\\
			A_{13}^TA_{31}+A_{23}^TA_{41} & A_{13}^TA_{33}+A_{23}^TA_{43}
		\end{pmatrix}.
	\end{equation}
\end{lemma}
\begin{proof}
	We use formula (41) in {\cite{CGshiftinvertible}}. By definition of metaplectic operator, for all $\tau\in\bR$, $z=(x,\xi)\in\rdd$,
	\begin{align*}
		\hat\cA(\rho(z;\tau)f\otimes \bar g)&=\hat\cA(\rho(x,0,\xi,0;\tau)f\otimes\bar g)\\
		&=\rho(\cA(x,0,\xi,0);\tau)\hat\cA(f\otimes\bar g)\\
		&=\rho(E_\cA z,F_\cA z;\tau)W_\cA(f,g).
	\end{align*}
	The assertion follows using that $\pi(x,\xi)=e^{i\pi x\cdot \xi}\rho(x,\xi;0)$:
	\begin{align*}
		W_\cA(\pi(x,\xi)f,g)&=W_\cA(e^{i\pi x\cdot \xi}\rho(x,\xi;0)f,g)\\
		&=e^{i\pi x\cdot \xi}\rho(E_\cA(x,\xi),F_\cA(x,\xi);0)W_\cA(f,g)\\
		&=e^{i\pi x\cdot \xi}e^{-i\pi E_\cA^TF_\cA(x,\xi)\cdot(x,\xi)}\pi(E_\cA(x,\xi),F_\cA(x,\xi))W_\cA(f,g).
	\end{align*}
	Using the definitions of $E_\cA$ and $F_\cA$, as well as the matrix $L$ in \eqref{defL}, so that we rewrite the scalar product as
	\[
		x\cdot \xi=L(x,\xi)\cdot(x,\xi),
	\]
	we infer
	\[
		e^{i\pi x\cdot \xi}e^{-i\pi E_\cA^TF_\cA(x,\xi)\cdot(x,\xi)}=e^{-i\pi M_\cA(x,\xi)\cdot(x,\xi)},
	\]
	where
	\[
			M_\cA=\begin{pmatrix}
			A_{11}^TA_{31}+A_{21}^TA_{41} & A_{11}^TA_{33}+A_{21}^TA_{43}-I_{d\times d}\\
			A_{13}^TA_{31}+A_{23}^TA_{41} & A_{13}^TA_{33}+A_{23}^TA_{43}
		\end{pmatrix}.
	\]
	The relations $(R1a)$, $(R2a)$ and $(R3a)$ imply that $M_\cA$ is symmetric and it can be written as in (\ref{defMA}).
	
\end{proof}

\begin{remark}
	We stress that (\ref{defMA}) introduces a new matrix associated to $W_\cA$. Throughout this work, if $E_\cA$ and $F_\cA$ are defined as in (\ref{defEAFA}), whereas $P$ is the matrix given in (\ref{defL}), $M_\cA$ denotes the symmetric $2d\times2d$ matrix defined as $M_\cA=E_\cA^TF_\cA-P$.
\end{remark}

\begin{theorem}\label{thm314}
	Let $\hat\cA\in Mp(2d,\bR)$, $\cA=\pi^{Mp}(\hat \cA)$ and $W_\cA$ be the associated metaplectic Wigner distribution. Consider the matrix $\cA_\ast\in Sp(2d,\bR)$ defined in \eqref{matrixAstar} below. Then,  for every $z\in\rdd$,
		\[
			\la \pi_\cA(z)f,g\ra=\la f,\pi_{\cA_\ast}(z)g\ra, \quad \forall \, f,g\in\cS(\rd).
		\]
		In particular, if $\pi_\cA(z)$ extends to a bounded operator on $L^2(\rd)$, then \begin{equation}\label{adjointA}
			\pi_\cA(z)^\ast=\pi_{\cA_\ast}(z),\quad z\in\rdd.
		\end{equation}
\end{theorem}
\begin{proof}
	It is an immediate consequence of Corollary \ref{corA} below. In fact, for all $f,g\in \cS(\rd)$,
	\begin{align*}
		\la \pi_\cA(z)f,g\ra&=\overline{W_\cA(g,f)(z)}=W_{\cA_\ast}(f,g)(z)=\overline{\la \pi_{\cA_\ast}(z) g,f \ra}=\la f,\pi_{\cA_\ast} (z)g\ra.
	\end{align*}
\end{proof}

\section{Shift-invertibility unmasked}\label{sec:SIU}

Among all metaplectic Wigner distributions, shift-invertible Wigner distributions are known to play a fundamental role in time-frequency analysis. It was proved in \cite{CR2022, CGshiftinvertible} that they can be used to replace the STFT in the definition of modulation spaces $M^{p,q}_m(\rd)$, for $1\leq p,q\leq\infty$ and $m\in\cM_v(\rdd)$ satisfying some inoffensive symmetry condition. In \cite{CGshiftinvertible} it is observed that shift-invertibility is necessary for this characterization to hold, otherwise not even the $M^p(\rd)$ spaces can be defined in terms of shift-invertible Wigner distributions. In this section, we investigate the properties of metaplectic atoms related to shift-invertible metaplectic Wigner distributions and characterize them in terms of the matrices $E_\cA$, $F_\cA$, $\cE_\cA$, $\cF_\cA$ and $M_\cA$ defined in (\ref{defEAFA}), (\ref{defeafa}) and (\ref{defMA}), respectively.\\

Take any metaplectic Wigner distribution $W_\cA$, and $z,w\in\rdd$. Then Lemma \ref{commpiWA} entails the equality
\[
W_\cA(\pi(w)f,g)(z)=\Phi_{-M_\cA}(w)\pi(E_\cA w,F_\cA w)W_\cA(f,g)(z),\qquad f,g\in L^2(\rdd),
\]
so that $|W_\cA(\pi(w)f,g)(z)|=|W_\cA(f,g)(z-E_\cA w)|$.

\begin{definition}
	A metaplectic Wigner distribution $W_\cA$ is \textbf{shift-invertible} if $E_\cA\in GL(2d,\bR)$.  
\end{definition}
We shall need the following \textit{lifting-type} result, proved in \cite[Theorem B1]{CGshiftinvertible}:
\begin{lemma}\label{lemma23}
	Let $\hat S_1,\hat S_2\in Mp(d,\bR)$ have block decompositions
	\[
	S_j=\begin{pmatrix}
		A_j & B_j\\
		C_j & D_j
	\end{pmatrix}
	\]
	($j=1,2$). Then, the bilinear operator
	\[
	T(f,g)=\hat S_1 f\otimes \hat S_2 g
	\]
	extends to a metaplectic operator $\hat S\in Mp(2d,\bR)$, where
	\begin{equation}\label{liftmatrix}
		S=\begin{pmatrix}
			A_1 & 0_{d\times d} & B_1 & 0_{d\times d}\\
			0_{d\times d} & A_2 & 0_{d\times d} & B_2\\
			C_1 & 0_{d\times d} & D_1 & 0_{d\times d}\\
			0_{d\times d} & C_2 & 0_{d\times d} & D_2
		\end{pmatrix}.
	\end{equation}
\end{lemma}

If $\hat S\in Mp(d,\bR)$ and $\hat T(f\otimes g)=f\otimes \hat Sg$, we set 
\begin{equation}\label{lift-def}
	\text{Lift}(S)=\pi^{Mp}(\hat T)\in Sp(2d,\bR),
\end{equation}
the corresponding matrix in (\ref{liftmatrix}).

\begin{theorem}\label{GiuLaMaschera}
	Let $W_\cA$ be a shift-invertible metaplectic Wigner distribution and $G_\cA=LE_\cA^{-1}\cE_\cA$ be the matrix of Lemma \ref{lemma44}, with  $L$  as in \eqref{defL}.
	 Then,
	\[
		\cA=\cD_{E_\cA^{-1}}V_{M_\cA}V_L^T \Lift(G_\cA),
	\]
	where $\Lift(G_\cA)$ is defined in \eqref{lift-def}.
\end{theorem}
\begin{proof}
	We use the matrix 
	\[
		\cK:=\begin{pmatrix}
			I_{d\times d} & 0_{d\times d} & 0_{d\times d} & 0_{d\times d}\\
			0_{d\times d} & 0_{d\times d} & I_{d\times d} & 0_{d\times d}\\
			0_{d\times d} & I_{d\times d} & 0_{d\times d} & 0_{d\times d}\\
			0_{d\times d} & 0_{d\times d} & 0_{d\times d} & I_{d\times d}
		\end{pmatrix},
	\]
	that permutes the central columns of $4d\times 4d$ matrices. This yields the following block decomposition of $\cA$:
	\[
		\cA=\begin{pmatrix} E_\cA & \cE_\cA \\ F_\cA & \cF_\cA \end{pmatrix}\cK.
	\]
	Since $E_\cA\in GL(2d,\bR)$, we can write
	\[
		\cA=\begin{pmatrix}
			E_\cA & 0_{2d\times 2d}\\
			0_{2d\times 2d} & E_\cA^{-T}
		\end{pmatrix}\begin{pmatrix}
		I_{2d\times 2d} & E_\cA^{-1}\cE_\cA\\
		E_\cA^T F_\cA & E_\cA^T \cF_\cA
		\end{pmatrix}\cK=\cD_{E_\cA^{-1}}\begin{pmatrix}
		I_{2d\times 2d} & E_\cA^{-1}\cE_\cA\\
		E_\cA^T F_\cA & E_\cA^T \cF_\cA
		\end{pmatrix}\cK.
	\]
We proved in Lemma \ref{commpiWA} that the matrix $M_\cA=E_\cA^T F_\cA -P$ is symmetric, where $P$ is defined as in (\ref{defL}). Therefore, $V_{M_\cA}$ is a symplectic matrix and we have:
	\begin{align*}
		\cA&=\cD_{E_\cA^{-1}}\begin{pmatrix}
			I_{2d\times 2d} & 0_{2d\times 2d}\\
			M_\cA & I_{2d\times 2d}
		\end{pmatrix}\begin{pmatrix}
			I_{2d\times 2d} & E_\cA^{-1}\cE_\cA\\
			P & E_\cA^T\cF_\cA - M_\cA E_\cA^{-1}\cE_\cA
		\end{pmatrix}\cK\\
		&=\cD_{E_\cA^{-1}} V_{M_\cA}\underbrace{\begin{pmatrix}
			I_{2d\times 2d} & E_\cA^{-1}\cE_\cA\\
			P & E_\cA^T\cF_\cA - M_\cA E_\cA^{-1}\cE_\cA
		\end{pmatrix}\cK}_\text{$=:\cA'$}.
	\end{align*}
	The matrix $\cA'$ is symplectic, since $\cA'=V_{-M_\cA}\cD_{E_\cA}\cA$ is the product of symplectic matrices. Getting rid of $\cK$, we obtain
	\[
		\cA'=\begin{pmatrix}
			I_{d\times d} & A_{12}' & 0_{d\times d} & A_{14}'\\
			0_{d\times d} & A_{22}' & I_{d\times d} & A_{24}'\\
			0_{d\times d} & A_{32}' & I_{d\times d} & A_{34}'\\
			0_{d\times d} & A_{42}' & 0_{d\times d} & A_{44}'
		\end{pmatrix},
	\]
	for suitable matrices $A_{ij}'$, $i=1,2,3,4$, $j=2,4$. Observe that
	\[
		E_{\cA}^{-1}\cE_{\cA}=\begin{pmatrix}
			A_{12}' & A_{14}'\\
			A_{22}' & A_{24}'
		\end{pmatrix}.
	\]
	The symplectic relations $(R1b)$, $(R1c)$, $(R2b)$, $(R2c)$, $(R3b)$, $(R3c)$ and $(R3d)$ for $\cA'\in Sp(2d,\bR)$ read respectively as
	\begin{align*}
			&(S1) \qquad \text{$A_{32}'=0_{d\times d}$},\\
			&(S2) \qquad \text{${A_{12}'}^TA_{32}'+{A_{22}'}^TA_{42}'={A_{32}'}^TA_{12}'+{A_{42}'}^TA_{22}'$},\\
			&(S3) \qquad \text{$A_{44}',=A_{14}'$},\\
			&(S4) \qquad \text{${A_{14}'}^T A_{34}'+{A_{24}'}^T A_{44}={A_{34}'}^T A_{14}+{A_{44}'}^T A_{24}$}\\
			&(S5) \qquad \text{$A_{34}'=0_{d\times d}$},\\
			&(S6) \qquad \text{$A_{12}'=A_{42}'$}\\
			&(S7) \qquad \text{${A_{12}'}^T A_{34}'+{A_{22}'}^T A_{44}'-({A_{32}'}^T A_{14}'+{A_{42}'}^T A_{24} )=I_{d\times d}$}.
	\end{align*}
	The others being trivially satisfied. This yields:
	\[
		\cA'=\begin{pmatrix}
			I_{d\times d} & A_{12}' & 0_{d\times d} & A_{14}'\\
			0_{d\times d} & A_{22}' & I_{d\times d} & A_{24}'\\
			0_{d\times d} & 0_{d\times d} & I_{d\times d} & 0_{d\times d}\\
			0_{d\times d} & A_{12}' & 0_{d\times d} & A_{14}'
		\end{pmatrix}.
	\]
	Observe that
	\begin{equation}\label{defGA}
		\begin{pmatrix}
			A_{22}' & A_{24}'\\
			A_{12}' & A_{14}'
		\end{pmatrix}=LE_\cA^{-1}\cE_\cA=G_\cA,
	\end{equation}
	which is symplectic by Lemma \ref{lemma44}. A simple computation shows that $\cA'=V_L^T \Lift(G_\cA)$, as desired.
\end{proof}



The characterization of shift-invertible Wigner distributions is straightforward.

\begin{corollary}\label{cor43}
	Let $W_\cA$ be a metaplectic Wigner distribution. Then, $W_\cA$ is shift-invertible if and only if, up to a sign,
	\begin{equation}\label{WA-STFT}
		W_\cA(f,g)=\mathfrak{T}_{E_\cA^{-1}}\Phi_{M_\cA+L}V_{\widehat{\delta_\cA} g}f,\quad f,g\in L^2(\rd),
	\end{equation}
	 where
	\begin{equation}\label{defSA}
		\widehat{\delta_\cA} g:=\cF\widehat{\overline{G_\cA}}g,
	\end{equation}
	and $\widehat{\overline{G_\cA}}$ is the metaplectic operator defined in Proposition \ref{propA2} below. In particular, if $W_\cA$ is shift-invertible then, up to a sign, 
	\begin{equation}\label{piASI}
		\pi_\cA(z)=|\det(E_\cA)|^{-1/2}\Phi_{-M_\cA-L}(E_\cA^{-1}z)\pi(E_\cA^{-1}z)\widehat{\delta_\cA},\quad z\in\rdd,
	\end{equation}
and\\
	(i) $\pi_\cA(z)$ is a surjective quasi-isometry of $L^2(\rd)$ with $$\norm{\pi_\cA(z)f}_{2}=|\det(E_\cA)|^{-1/2}\norm{f}_2, \quad f\in L^2(\rd);$$
	(ii) $\pi_\cA(z)$ is a topological isomorphism on $\cS(\rd)$;\\
	(iii) $\pi_\cA(z)$ is a topological isomorphism on $\cS'(\rd)$. 
\end{corollary}
\begin{proof}
	By Theorem \ref{GiuLaMaschera}, $\cA$ is shift-invertible if and only if
	\[
	\cA=\cD_{E_\cA^{-1}}V_{M_\cA}V_L^T\Lift(G_\cA).
	\]
	Let $A_{ST}$ be the symplectic matrix associated to the STFT, cf. (\ref{AST}). Observe that
	\[
		A_{ST}=V_{-L}V_L^T\cA_{FT2},
	\]
	where $\cA_{FT2}$ is the symplectic matrix associated to the partial Fourier transform with respect to the second variable defined in  (\ref{AFT2}). Then, 
	\[
		\cA=\cD_{E_\cA^{-1}}V_{M_\cA}(V_{L}V_{-L})V_L^T(\cA_{FT2}\cA_{FT2}^{-1})\Lift(G_\cA)=\cD_{E_\cA^{-1}}V_{M_\cA+L}A_{ST}\cA_{FT2}^{-1}\Lift(G_\cA).
	\]
	Therefore, up to a sign, 
	\begin{align*}
		W_\cA(f,g)(z)&=\hat\cA(f\otimes \bar g)(z)=\widehat{\cD_{E_\cA^{-1}}}\widehat{ V_{M_\cA}}\widehat{ V_L^T}\widehat{\Lift(G_\cA)}(f\otimes\bar g)(z)\\
		&=\widehat{\cD_{E_\cA^{-1}}}\widehat{ V_{M_\cA+L}}\widehat{A_{ST}}\cF_2^{-1}\widehat{\Lift(G_\cA)}(f\otimes\bar g)(z)\\
		&=|\det(E_\cA)|^{-1/2}\Phi_{M_\cA+L}(E_\cA^{-1}z)\widehat{A_{ST}}(f\otimes (\cF^{-1}\widehat{G_\cA}\bar g))(E_\cA^{-1}z).
	\end{align*}
	Let $\widehat{\overline{ G_\cA}}$ be the symplectic operator such that $\widehat{G_\cA} \bar g=\overline{\widehat{\overline{ G_\cA}}g}$, cf. Proposition \ref{propA2}. Then, 
	\[
		\cF^{-1}\widehat{G_\cA}\bar g=\cF^{-1}\overline{\widehat{\overline{ G_\cA}}g}=\overline{\cF\widehat{\overline{ G_\cA}} g}=:\overline{\widehat{\delta_\cA} g}.
	\]
	Therefore,
	\[
		W_\cA(f,g)(z)=|\det(E_\cA)|^{-1/2}\Phi_{M_\cA+L}(E_\cA^{-1}z)V_{\widehat{\delta_\cA} g}f(E_\cA^{-1}z),
	\] 
	which can also be restated as:
	\[
		W_\cA(f,g)(z)=\la f,\pi_\cA(z)g \ra,
	\]
	where $\pi_\cA(z)$ is the operator in (\ref{piASI}). Items $(i)$ - $(iii)$ are trivial consequences of (\ref{piASI}).
\end{proof}

The metaplectic operator defined in (\ref{defSA}) plays a crucial role in the characterization of metaplectic Gabor frames for shift-invertible metaplectic Wigner distributions. For this reason, it is worth giving it a name, in the spirit of the terminology used by M. de Gosson in \cite{DGosson}:

\begin{definition}\label{defop}
We call the metaplectic operator $\widehat{\delta_\cA}$ in (\ref{defSA}) the \textbf{deformation operator} associated to $W_\cA$.
\end{definition}

\begin{example} $\tau$-Wigner distributions can be rephrased as \emph{rescaled  STFT, up to chirps}, as in \eqref{WA-STFT}. Precisely, for $0<\tau<1$, set $\mathfrak{T}_\tau g(t)=\frac{(1-\tau)^{d/2}}{\tau^{d/2}}g(-\frac{1-\tau}{\tau}t)$ as in Example \ref{example34}. We proved in the same Example that
\begin{equation}\label{defPiWt}
	W_\tau(f,g)(x,\xi)=\left\langle f,\frac{1}{\tau^{d/2}(1-\tau)^{d/2}}e^{-2\pi i\frac{x\cdot \xi}{\tau}}\pi\left(\frac{x}{1-\tau},\frac{\xi}{\tau}\right)\mathfrak{T}_\tau g\right\rangle
\end{equation}
for all $f,g\in L^2(\rd)$ and $x,\xi\in\rd$. Consequently, we retrieve the expression of $W_\tau$ as a rescaled STFT:
\begin{align*}
	W_\tau(f,g)(x,\xi)=\frac{1}{\tau^{d/2}(1-\tau)^{d/2}}e^{2\pi i\frac{x\cdot \xi}{\tau}}V_{\mathfrak{T}_\tau g}f\left(\frac{x}{1-\tau},\frac{\xi}{\tau}\right).
\end{align*}
	\end{example}
We proved that metaplectic atoms of shift-invertible Wigner distributions are surjective isometries of $L^2(\rd)$ and their adjoints are the atoms associated to $W_{\cA_\ast}$, where $\cA_\ast$ is the matrix defined in the statement of the Theorem \ref{thm314}.

We conclude this section with the explicit computation of $\pi_\cA(z)^{-1}$ and $\pi_\cA(z)^\ast$ for shift-invertible Wigner distributions.

\begin{theorem}\label{thm46}
	Let $W_\cA$ be a shift-invertible Wigner distribution and $\widehat{\delta_\cA}$  the related deformation operator, cf. \eqref{defSA}. Consider the matrices  $L$ and $P$  defined as in (\ref{defL}) and the following matrices:
	 $$Q=\begin{pmatrix}I_{d\times d} & 0_{d\times d}\\ 0_{d\times d} & -I_{d\times d}\end{pmatrix}=-LJ,$$
	\begin{equation}\label{esplTA}
		\delta_\cA=-E_\cA^{-1}\cE_\cA Q.
	\end{equation}
	
	 Then, for every $z\in\rdd$, up to a sign,  the inverse $\pi_\cA(z)^{-1}$ and the adjoint $\pi_\cA(z)^{\ast}$  operators can be explicitly computed as
 \begin{equation}\label{pi-inverso}
		\pi_\cA(z)^{-1}=|\det(E_\cA)|^{1/2}\Phi_{M_\cA+L/2}(E_\cA^{-1}z)\Phi_{L/2}(\cE_\cA^{-1}z)\pi(Q\cE_\cA^{-1}z)\widehat{\delta_\cA}^{-1},
\end{equation}
and
	\begin{equation}\label{pi-aggiunto}
	\pi_\cA(z)^{\ast}=|\det(E_\cA)|^{-1}\pi_\cA(z)^{-1}.
	\end{equation}
\end{theorem}
\begin{proof}
	We use the explicit expression of metaplectic Gabor atoms for shift-invertible $W_\cA$ in (\ref{piASI}), which yields
	\begin{equation}\label{eq461}
		\pi_\cA(z)^{-1}=|\det(E_\cA)|^{1/2}\Phi_{M_\cA+L}(E_\cA^{-1}z)\widehat{\delta_\cA}^{-1}\pi(E_\cA^{-1}z)^{-1}.
	\end{equation}
	By (\ref{inverseTFs}), if $E_\cA^{-1}z=((E_\cA^{-1}z)_1,(E_\cA^{-1}z)_2)$,
	\[
		\pi(E_\cA^{-1}z)^{-1}=e^{-2\pi i(E_\cA^{-1}z)_1\cdot(E_\cA^{-1}z)_2}\pi(-E_\cA^{-1}z)=\Phi_{-L}(E_\cA^{-1}z)\pi(-E_\cA^{-1}z).
	\]
	Also, by (\ref{muAdef}), for all $z\in\rdd$ and $\tau\in\bR$,
	\[
		\widehat{\delta_\cA}^{-1}\rho(-E_\cA^{-1}z;\tau)\widehat{\delta_\cA}=\rho(-\delta_\cA^{-1}E_\cA^{-1}z;\tau).
	\]
	Using the definition of $\rho$, for $\tau=0$ this is equivalent to
	\begin{equation}\label{eq462}
		\widehat{\delta_\cA}^{-1}\pi(-E_\cA^{-1}z)=e^{i\pi (E_\cA^{-1}z)_1\cdot(E_\cA^{-1}z)_2}e^{-i\pi(\delta_\cA^{-1}E_\cA^{-1}z)_1\cdot(\delta_\cA^{-1}E_\cA^{-1}z)_2}\pi(-\delta_\cA^{-1}E_\cA^{-1}z)\widehat{\delta_\cA}^{-1},
	\end{equation}
	where $\delta_\cA^{-1}E_\cA^{-1}z=((\delta_\cA^{-1}E_\cA^{-1}z)_1,(\delta_\cA^{-1}E_\cA^{-1}z)_2)$. 
	We compute explicitly the matrix $\delta_\cA^{-1}E_\cA^{-1}$. For, let us denote with
	\[
		G_\cA=\begin{pmatrix}
		A & B\\
		C & D
		\end{pmatrix}
	\]
	the block decomposition of the symplectic matrix $G_\cA$, so that
	\[
		G_\cA^{-1}=\begin{pmatrix}
		D^T & -B^T\\
		-C^T & A^T
		\end{pmatrix}, \quad \overline{G_\cA}= \begin{pmatrix}
			A & -B\\
			-C & D
		\end{pmatrix}\quad and \quad \overline{G_\cA^T}=\overline{G_\cA}^T=\begin{pmatrix}
			A^T & -C^T \\
			-B^T & D^T
		\end{pmatrix}.
	\]
	By definition, $\delta_\cA=\pi^{Mp}(\widehat{\delta_\cA})=\pi^{Mp}(\cF \widehat{\overline{G_\cA}})$, so that
	\[
		\delta_\cA=J\overline{G_\cA}.
	\]
	This, together with $\overline{G_\cA}J\overline{G_\cA}^T=J$ and $G_\cA=LE_\cA^{-1}\cE_\cA$, yields to:
	\begin{align*}
		\delta_\cA^{-1}E_\cA^{-1}&=(-\overline{G_\cA}^{-1}J)(LG_\cA\cE_\cA^{-1})=(-J\overline{G_\cA}^T)(LG_\cA\cE_\cA^{-1}),
	\end{align*}
	where the invertibility of $\cE_\cA$ is guaranteed by Lemma \ref{lemma44}. We use the block decompositions of the matrices at stake to get:
	\begin{align*}
		\delta_\cA^{-1}E_\cA^{-1}&=\begin{pmatrix}
		0_{d	\times d} & -I_{d\times d}\\
		I_{d\times d} & 0_{d\times d}
		\end{pmatrix}\begin{pmatrix}
			A^T & -C^T\\
			-B^T & D^T
		\end{pmatrix}\begin{pmatrix}
			0_{d\times d} & I_{d\times d}\\
			I_{d\times d} & 0_{d\times d}
		\end{pmatrix}G_\cA\cE_\cA^{-1}\\
		&=\begin{pmatrix}
			B^T & -D^T\\
			A^T & -C^T
		\end{pmatrix}\begin{pmatrix}
			0_{d\times d} & I_{d\times d}\\
			I_{d\times d} & 0_{d\times d}
		\end{pmatrix}G_\cA\cE_\cA^{-1}\\
		&=\begin{pmatrix}
			-D^T & B^T\\
			-C^T & A^T
		\end{pmatrix}G_\cA\cE_\cA^{-1}\\
		&=\begin{pmatrix}
			-I_{d\times d} & 0_{d\times d}\\
			0_{d\times d} & I_{d\times d}
		\end{pmatrix}G_\cA^{-1}G_\cA\cE_\cA^{-1}=-Q \cE_\cA^{-1}.
	\end{align*}
	This proves (\ref{esplTA}). A simple computation shows that 
	\begin{equation}\label{deltaAEAcEA}
		(\delta_\cA^{-1}E_\cA^{-1}z)_1\cdot(\delta_\cA^{-1}E_\cA^{-1}z)_2=(Q\cE_\cA^{-1}z)_1\cdot(Q\cE_\cA^{-1}z)_2=-(\cE_\cA^{-1}z)_1\cdot(\cE_\cA^{-1}z)_2,
	\end{equation}
	that entails
	\[
		e^{-i\pi(\delta_\cA^{-1}E_\cA^{-1}z)_1\cdot(\delta_\cA^{-1}E_\cA^{-1}z)_2}=e^{i\pi (\cE_\cA^{-1}z)_1\cdot (\cE_\cA^{-1}z)_2}=\Phi_{L/2}(\cE_\cA^{-1}z).
	\]
	Plugging all the information in (\ref{eq461}), we find
	\[
		\pi_\cA(z)^{-1}=|\det(E_\cA)|^{1/2}\Phi_{M_\cA+L}(E_\cA^{-1}z)\Phi_{-L/2}(E_\cA^{-1}z)\Phi_{L/2}(\cE_\cA^{-1}z)\pi(Q\cE_\cA^{-1}z)\widehat{\delta_\cA}^{-1}.
	\]

	This proves $(i)$.

	To prove $(ii)$, we prove that $\pi_\cA(z)^\ast$ is expressed by (\ref{eq461}), up to the determinant factor. For, let $f,g\in L^2(\rd)$ and $z\in\rdd$. By (\ref{piASI}),
	\begin{align*}
		\la\pi_\cA(z)^\ast f,g\ra &=\la f,\pi_\cA(z)g\ra\\
		&=\la f,|\det(E_\cA)|^{-1/2}\Phi_{-M_\cA-L}(E_\cA^{-1}z)\pi(E_\cA^{-1}z)\widehat{\delta_\cA}g\ra\\
		&=\la |\det(E_\cA)|^{-1/2}\Phi_{M_\cA+L}(E_\cA^{-1}z)\widehat{\delta_\cA}^{-1}\pi(E_\cA^{-1}z)^{-1}f,g\ra\\
		&=\la |\det(E_\cA)|^{-1}\pi_\cA(z)^{-1}f,g\ra
	\end{align*}
	and the assertion follows.
\end{proof}

\section{Atoms of Covariant Metaplectic Wigner distributions}\label{sec:CA}
In this section we derive the expression of metaplectic atoms of covariant metaplectic Wigner distributions. We recall their definition, cf. \cite{CR2022}
\begin{definition}
	A metaplectic Wigner distribution $W_\cA$ is \textbf{covariant} if
	\[
		W_\cA(\pi(z)f,\pi(z)g)=T_zW_\cA(f,g)
	\]
	holds for every $z\in\rdd$ and all $f,g\in L^2(\rd)$.
\end{definition}
The following result summarizes \cite[Proposition 2.10 and Theorem 2.11]{CR2022} and states that covariance characterises the Cohen's class of metaplectic Wigner distributions. 

\begin{proposition}\label{propCarCov}
	Let $\hat\cA\in Mp(2d,\bR)$ and $W_\cA$ be the associated metaplectic Wigner distribution. The following statements are equivalent:\\
	(i) $W_\cA$ is covariant.\\
	(ii) The matrix $\cA=\pi^{Mp}(\hat\cA)$ has block decomposition 
	\begin{equation}\label{defAcov}
		\cA=\begin{pmatrix}
			A_{11} & I_{d\times d}-A_{11} & A_{13} & A_{13}\\
			A_{21} & -A_{21} & I_{d\times d} -A_{11}^T & -A_{11}^T\\
			0_{d\times d} & 0_{d\times d} & I_{d\times d} & I_{d\times d} \\
			-I_{d\times d}  & I_{d\times d}  & 0_{d\times d}  & 0_{d\times d} 
		\end{pmatrix},
	\end{equation}
	with $A_{13}=A_{13}^T$ and $A_{21}=A_{21}^T$.\\
	(iii) $W_\cA$ belongs to the Cohen's class, namely 
	\begin{equation}
		W_\cA(f,g)=\Sigma_\cA\ast W(f,g), \qquad f,g\in L^2(\rd),
	\end{equation}
	where $\Sigma_\cA=\cF^{-1}\Phi_{-B_\cA}$, with $B_\cA$ defined as in (\ref{defBA}). 
\end{proposition}

\begin{theorem}
	Let $W_\cA$ be a covariant metaplectic Wigner distribution, $\cA$ and $B_\cA$ be as in (\ref{defAcov}) and (\ref{defBA}), respectively. Then, \\
	(i) for every $z\in\rdd$,
	\begin{equation}\label{expthmcov1}
		\pi_\cA(z)g\overset{\cS'}{=}2^d\int_{\rdd}\cF\Phi_{B_\cA}(z-w) \Phi_{-2L}(w) \pi(2w)\mathcal{I}g  dw,
	\end{equation}
	where $\mathcal{I}g(t)=g(-t)$ and the integral must be interpreted in the weak sense of vector-valued integration. \\
	(ii) If also $B_\cA\in GL(2d,\bR)$, then, for every $z\in\rdd$,
	\begin{equation}\label{expthmcov2}
		\pi_\cA(z)g\overset{\cS'}{=}2^d\int_{\rdd}\Phi_{-B_\cA^{-1}}(z-w) \Phi_{-2L}(w) \pi(2w)\mathcal{I}g  dw, \qquad g\in \cS(\rd)
	\end{equation}	
	holds in the weak sense of vector-valued integration.\\
	(iii) If $\cA_\ast$ is the matrix defined in (\ref{matrixAstar}), then $W_{\cA_\ast}$ is covariant with $B_{\cA_\ast}=-B_\cA$ and, consequently,
	\[
		\pi_{\cA}(z)^\ast g \overset{\cS'}{=}2^d\int_{\rdd}\cF\Phi_{-B_\cA}(z-w) \Phi_{-2L}(w) \pi(2w)\mathcal{I}g  dw, 
	\]
	for all $g\in\cS(\rd)$ and every $z\in\rdd$. If $B_\cA$ is invertible, then
	\[
		\pi_\cA(z)^\ast g\overset{\cS'}{=}2^d\int_{\rdd}\Phi_{B_\cA^{-1}}(z-w) \Phi_{-2L}(w) \pi(2w)\mathcal{I}g  dw,
	\]
	for every $g\in \cS(\rd)$ and $z\in\rdd$.
\end{theorem}
\begin{proof}
	$(i)$ By Proposition \ref{propCarCov}, for all $\varphi,g\in \cS(\rd)$ and all $z\in\rdd$, 
	\begin{align*}
		\la \varphi,\pi_\cA(z)g\ra &={W_\cA(\varphi,g)(z)}\\
		&={\Sigma_\cA\ast W(\varphi,g)(z)}\\
		&=\int_{\rdd}{\Sigma_{\cA}(z-w)}  {W(\varphi,g)(w)}dw\\
		&=\int_{\rdd}\overline{\cF\Phi_{B_\cA}(z-w)}\la\varphi, \pi_{A_{1/2}}(w)g\ra dw\\
		&=\left\la \varphi, \int_{\rdd}\cF\Phi_{B_\cA}(z-w)\pi_{A_{1/2}}(w)gdw \right\ra,
	\end{align*}
	where we used that $\cF^{-1}\Phi_{-B_{\cA}}=\overline{\cF\Phi_{B_\cA}}$. Consequently,
	\[
		\pi_\cA(z)g=\int_{\rdd}\cF\Phi_{B_\cA}(z-w)\pi_{A_{1/2}}(w)gdw.
	\]
	Plugging $\tau=1/2$ in (\ref{defPiWt}), we infer the explicit metaplectic atom  of  the Wigner distribution: for $w=(x,\xi)\in\rdd$,
	\[
		\pi_{A_{1/2}}(x,\xi)g(t)=2^de^{-4\pi ix\cdot \xi}\pi(2x,2\xi)\mathcal{I}g(t)=2^d\Phi_{-2L}(w)\pi(2w)\mathcal{I}g(t).
	\]
	Expression (\ref{expthmcov1}) follows consequently. \\
	
	$(ii)$ If $B_\cA$ is invertible, then $\cF\Phi_{B_\cA}=\Phi_{-B_\cA^{-1}}$, and (\ref{expthmcov2}) holds in the weak sense of vector-valued integration. 
	
	$(iii)$ By (\ref{matrixAstar}) and (\ref{defAcov}), it follows that
	\[
		\cA_\ast=\begin{pmatrix}
			I_{d\times d}-A_{11} & A_{11} & -A_{13} & -A_{13}\\
			-A_{21} & A_{21} & A_{11}^T & A_{11}^T-I_{d\times d}\\
			0_{d\times d} & 0_{d\times d} & I_{d\times d} & I_{d\times d}\\
			-I_{d\times d} & I_{d\times d} & 0_{d\times d} & 0_{d\times d}
		\end{pmatrix}.
	\]	
	Therefore, $W_{\cA_\ast}$ is covariant by Proposition \ref{propCarCov} $(ii)$, with
	\begin{align*}
		B_{\cA_\ast}&=\begin{pmatrix}
			-A_{13} & \frac{1}{2}I_{d\times d} - (I_{d\times d}-A_{11})\\
			\frac{1}{2}I_{d\times d}-(I_{d\times d}-A_{11})^T & A_{21}
		\end{pmatrix}\\
		&=\begin{pmatrix}
			-A_{13} & A_{11}-\frac{1}{2}I_{d\times d}\\
			A_{11}^T-\frac{1}{2}I_{d\times d} & A_{21}
		\end{pmatrix}\\
		&=-B_\cA.
	\end{align*}
	So, $(iii)$ follows by $(i)$ and $(ii)$.
\end{proof}
%
%
%

\section{Metaplectic Gabor frames}\label{sec:MGF}
%

\begin{definition}\label{defMGF}
	Let $W_\cA$ be a metaplectic Wigner distribution such that every $\pi_\cA(z)$ extends to a bounded operator on $L^2(\rd)$ ($z\in\rdd$). Let $g\in L^2(\rd)$ and $\Lambda\subset\rdd$ be a discrete subset. We call the set $$\cG_\cA(g,\Lambda)=\{\pi_\cA(\lambda)g\}_{\lambda\in \Lambda}$$ a \textbf{metaplectic Gabor system}. We call \textbf{metaplectic Gabor frame} (of $L^2(\rd)$) any metaplectic Gabor system $\cG_\cA(g,\Lambda)$ such that the following property holds: there exist $A,B>0$ such that
	\begin{equation}\label{ineqMGF}
		A\norm{f}_2^2\leq\sum_{\lambda\in\Lambda}|W_\cA(f,g)(\lambda)|^2\leq B\norm{f}_2^2,
	\end{equation}
	 for all $ f\in L^2(\rd)$.
\end{definition}

\begin{remark}
	By Definition \ref{def1}, (\ref{ineqMGF}) is equivalent to
	\[
		A\norm{f}_2^2\leq\sum_{\lambda\in\Lambda}|\langle f,\pi_\cA(\lambda)g \rangle|^2\leq B\norm{f}_2^2,\quad\forall f\in\lrd.
	\]
	Stated differently, a metaplectic Gabor frame is a frame for $L^2(\rd)$.
\end{remark}

\begin{example}\label{exDG2}
	In \cite{DGosson}, M. de Gosson introduced $\hbar$-Gabor frames as follows. Consider $g\in L^2(\rd)$ and $\Lambda$ a discrete subset of $\rdd$. Under the same notation of Example \ref{exDG}, a family $\cG_\hbar(g,\Lambda)=\{\pi^\hbar(\lambda)g\}_{\lambda\in\Lambda}$ is a \textbf{$\hbar$-Gabor frame} if 
	\[
		A\norm{f}_2^2\leq\sum_{\lambda\in\Lambda}|\la f,\pi^\hbar (\lambda)g\ra|^2\leq B\norm{f}_2^2, \qquad \forall f\in L^2(\rd),
	\]
	for $A,B>0$. The time-frequency representation $z\mapsto \la f,\pi^\hbar (z)g\ra$ is, up to the constant $(2\pi \hbar)^{-d/2}$ (which is necessary to obtain a metaplectic operator in Example \ref{exDG}), the metaplectic Wigner distribution $V^\hbar$, as defined in Example \ref{exDG}. Hence, metaplectic Gabor frames $\cG_{\cA_{\hbar}}$ and $\hbar$-Gabor frames are basically the same objects. Namely, $\cG_\hbar(g,\Lambda)$ is a $\hbar$-Gabor frame with frame bounds $A,B$ if and only if $\cG_{\cA_\hbar}(g,\Lambda)$ is a metaplectic Gabor frame with frame bounds $(2\pi\hbar)^{-d}A$ and $(2\pi\hbar)^{-d}B$.
\end{example}

Metaplectic Gabor frames associated to shift-invertible Wigner distributions are completely characterized by the following consequence of Corollary \ref{cor43}.

\begin{theorem}\label{thmFrames}
	Let $W_\cA$ be shift-invertible and $\widehat{\delta_\cA}$ be the corresponding deformation operator (see Definition \ref{defop}). Let $g\in L^2(\rd)$ and $\Lambda\subseteq\rdd$ be a discrete subset. The following statements are equivalent:\\
	(i) $\cG_\cA(g,\Lambda)$ is a metaplectic Gabor frame with bounds $A,B$;\\
	(ii) $\cG(\widehat{\delta_\cA}g,E_\cA^{-1}\Lambda)$ is a Gabor frame with bounds $|\det(E_\cA)|A,|\det(E_\cA)|B$;\\
	(iii) $\cG(g,-Q\cE_\cA^{-1}\Lambda)$ is a Gabor frame with bounds $|\det(E_\cA)|A,|\det(E_\cA)|B$.
\end{theorem}
\begin{proof}
	Consider $f\in L^2(\rd)$. We use the representation of $\pi_{\cA}$ in (\ref{piASI}):
	\begin{align*}
		\sum_{\lambda\in\Lambda}|\la f,\pi_\cA(\lambda)g\ra|^2&=\sum_{\lambda\in\Lambda}|\la f,|\det(E_\cA)|^{-1/2}\pi(E_\cA^{-1}\lambda)\widehat{\delta_\cA}g\ra|^2\\
		&=|\det(E_\cA)|^{-1}\sum_{\mu\in E_\cA^{-1}\Lambda}|\la f,\pi(\mu)\widehat{\delta_\cA}g|^2.
	\end{align*}
	This proves the equivalence $(i)\Leftrightarrow(ii)$. Now, using (\ref{muAdef}), we can write
	\begin{align*}
		|\det(E_\cA)|^{-1}\sum_{\mu\in E_\cA^{-1}\Lambda}|\la f,\pi(\mu)\widehat{\delta_\cA}g|^2&=|\det(E_\cA)|^{-1}\sum_{\mu\in E_\cA^{-1}\Lambda}|\la f, \widehat{\delta_\cA}\pi(\delta_\cA^{-1}  \mu)g\ra|^2\\
		&=|\det(E_\cA)|^{-1}\sum_{\mu\in E_\cA^{-1}\Lambda}|\la \widehat{\delta_\cA}^{-1}f,\pi({\delta_\cA}^{-1}\mu)g\ra|^2\\
		&=|\det(E_\cA)|^{-1}\sum_{\nu\in {\delta_\cA}^{-1}E_\cA^{-1}\Lambda}|\la \widehat{\delta_\cA}^{-1}f,\pi(\nu)g\ra|^2.
	\end{align*}
	Observing that $\delta_\cA^{-1}E_\cA^{-1}=-Q\cE_\cA^{-1}$, 
	\begin{align*}
		|\det(E_\cA)|^{-1}\sum_{\mu\in E_\cA^{-1}\Lambda}|\la f,\pi(\mu)\widehat{\delta_\cA}g|^2&=|\det(E_\cA)|^{-1}\sum_{\nu\in -Q\cE_\cA^{-1}\Lambda}|\la \delta_\cA^{-1}f,\pi(\nu)g \ra|^2.
	\end{align*}
	Therefore, $\mathcal{G}_\cA(g,\Lambda)$ is a metaplectic Gabor frame with frame bounds $A$ and $B$ if and only if 
	\begin{equation}\label{equiv1}
		A\norm{f}_2^2\leq|\det(E_\cA)|^{-1}\sum_{\mu\in -Q\cE_\cA^{-1}\Lambda}|\la \delta_\cA^{-1}f,\pi(\mu)g \ra|^2\leq B\norm{f}_2^2, \qquad f\in L^2(\rd).
	\end{equation}
	Since $\widehat{\delta_\cA}^{-1}$ is a unitary operator on $L^2(\rd)$, it follows that (\ref{equiv1}) holds for all $f\in L^2(\rd)$ if and only if
	\[
		|\det(E_\cA)|A\norm{f}_2^2\leq\sum_{\mu\in -Q\cE_\cA^{-1}\Lambda}|\la f,\pi(\mu)g \ra|^2\leq|\det(E_\cA)| B\norm{f}_2^2
	\]
	holds for every $f\in L^2(\rd)$. This proves the equivalence $(i)\Leftrightarrow(iii)$.
\end{proof}
\begin{remark}
	For $\hbar$-Gabor frames, Example \ref{exDG2} shows that Theorem \ref{thmFrames} applied to the metaplectic Wigner distributions of Example \ref{exDG} recovers \cite[Proposition 7]{DGosson}.
\end{remark}

\vspace{0.5truecm}
We now introduce the metaplectic Gabor frame operator and related properties.

First, consider a lattice $\Lambda\subset \rdd$ and  a metaplectic Gabor frame $\cG_\cA(g,\Lambda)=\{\pi_\cA(\lambda)g\}_{\lambda\in \Lambda}$ for $\lrd$. We compute the expressions of  coefficient, reconstruction and frame operators, see, e.g.,  \cite[Definitions 3.1.8 and 3.1.13]{Elena-book}. The coefficient (or analysis) operator $C_{\cA}:\lrd\to \ell^2(\Lambda)$ is given by
\begin{equation}\label{coeffoper}
	C_{\cA} f =(\la f, \pi_\cA(\lambda)g \ra)_{\lambda\in\Lambda}=(W_\cA (f, g)(\lambda))_{\lambda\in\Lambda},\quad f\in\lrd.
\end{equation}
Its adjoint $D_\cA=C_{\cA}^*: \ell^2(\Lambda)\to \lrd$ is called the reconstruction (or synthesis) operator: for any sequence $c=(c_\lambda)_{\lambda\in\Lambda}\in\ell^2(\Lambda)$,
\begin{equation}\label{synthop}
	D_{\cA} c =\sum_{\lambda\in\Lambda}c_\lambda \pi_\cA(\lambda)g.
\end{equation}
The frame operator is defined as $S_\cA=D_\cA C_\cA:\lrd\to\lrd$:
\begin{equation}\label{frameop}
	S_\cA f=\sum_{\lambda\in\Lambda}\la f, \pi_\cA(\lambda)g \ra \pi_\cA(\lambda)g= \sum_{\lambda\in\Lambda}W_\cA (f, g)(\lambda)\pi_\cA(\lambda)g.
\end{equation}

Let us compute $\pi_\cA(\mu)^{-1} S_\cA \pi_\cA(\mu)$, for $\mu\in\Lambda$. We make use of the explicit expression of the inverse and the adjoint of the metaplectic atom \eqref{piASI} in \eqref{pi-inverso}, and \eqref{pi-aggiunto}, respectively.  Observing that the phase factors cancel, we obtain
\begin{align*}
	\pi_\cA(\mu)^{-1}S_\cA \pi_\cA(\mu)f&=\sum_{\lambda\in\Lambda} \la \pi_\cA(\mu)f,\pi_\cA(\lambda)g\ra\pi_\cA(\mu)^{-1}\pi_\cA(\lambda)g\\
	&=\sum_{\lambda\in\Lambda} \la f,\pi_\cA(\mu)^\ast\pi_\cA(\lambda)g\ra\pi_\cA(\mu)^{-1}\pi_\cA(\lambda)g\\
	&=|\det (E_\cA)|^{-1}\sum_{\lambda\in\Lambda} \la f,\pi_\cA(\mu)^{-1}\pi_\cA(\lambda)g\ra\pi_\cA(\mu)^{-1}\pi_\cA(\lambda)g\\
	&=|\det (E_\cA)|^{-1}\sum_{\lambda\in\Lambda} \la f,\widehat{\delta_\cA}^{-1}\pi(E_\cA^{-1}\mu)^{-1}\pi(E_\cA^{-1}\lambda)\widehat{\delta_\cA} g\ra\\
	&\qquad\qquad\times\widehat{\delta_\cA}^{-1}\pi(E_\cA^{-1}\mu)^{-1}\pi(E_\cA^{-1}\lambda)
	\widehat{\delta_\cA}g\\
	&=|\det (E_\cA)|^{-1}\sum_{\lambda\in\Lambda} \la f,\widehat{\delta_\cA}^{-1}\pi(E_\cA^{-1}(\lambda-\mu))\widehat{\delta_\cA} g\ra
	\widehat{\delta_\cA}^{-1}\pi(E_\cA^{-1}(\lambda-\mu))\widehat{\delta_\cA}g\\
	&=|\det (E_\cA)|^{-1}\sum_{\lambda\in\Lambda} \la f,\widehat{\delta_\cA}^{-1}\pi(E_\cA^{-1}\lambda)\widehat{\delta_\cA} g\ra
	\widehat{\delta_\cA}^{-1}\pi(E_\cA^{-1}\lambda)\widehat{\delta_\cA}g\\
	&=\sum_{\lambda\in\Lambda} \la f,\widehat{\delta_\cA}^{-1}\pi_\cA(\lambda) g\ra
	\widehat{\delta_\cA}^{-1}\pi_\cA(\lambda)g\\
	&=\sum_{\lambda\in\Lambda} \la  \widehat{\delta_\cA} f,\pi_\cA(\lambda) g\ra
	\widehat{\delta_\cA}^{-1}\pi_\cA(\lambda)g\\
	&=	\widehat{\delta_\cA}^{-1}S_\cA \widehat{\delta_\cA} f,
\end{align*}
since $\widehat{\delta_\cA}^{-\ast}= \widehat{\delta_\cA}.$

The equality $$	\pi_\cA(\mu)^{-1}S_\cA=\widehat{\delta_\cA}^{-1}S_\cA \widehat{\delta_\cA}\pi_\cA(\mu)^{-1}$$
yields
$$S_\cA^{-1} \pi_\cA(\mu)=\pi_\cA(\mu)\widehat{\delta_\cA}^{-1}S_\cA^{-1} \widehat{\delta_\cA}.$$
Hence the canonical dual frame of $\cG_\cA(g,\Lambda)$ is still a metaplectic Gabor frame 
\begin{equation}\label{canonic-dual}
	\cG_\cA(\gamma_\cA,\Lambda)=\{	\pi_\cA(\lambda)\gamma_\cA\}_{\lambda\in\Lambda}
\end{equation}
with canonical dual window
\begin{equation}\label{dualecanonica}
	\gamma_\cA= \widehat{\delta_\cA}^{-1}S_\cA^{-1} \widehat{\delta_\cA} g.
\end{equation}

Consequently, if $\cG_\cA(g,\Lambda)$ is a frame with bounds $0<A\leq B$, then every $f\in\lrd$ possesses the expansions
\begin{align}\label{C3Gaborexp}
	f&=\sum_{\lambda\in\Lambda} \la f,\pi_\cA(\lambda)g\ra \pi_\cA(\lambda)\gamma_\cA\\
	&=\sum_{\lambda\in\Lambda} \la f,\pi_\cA(\lambda)\gamma_\cA\ra \pi_\cA(\lambda)g
\end{align}
with unconditional convergence in $\lrd$. Besides, we have the norm equivalences
\begin{align*}
	A\|f\|_2^2\leq& \sum_{\lambda\in\Lambda} |\la f,\pi_\cA(\lambda)g\ra|^2\leq B\|f\|^2 \\
	B^{-1}\|f\|^2\leq & \sum_{\lambda\in\Lambda} |\la f,\pi_\cA(\lambda)\gamma_\cA\ra|^2\leq 	A^{-1}\|f\|_2^2.
\end{align*}

\section{Characterization of Time-frequency spaces }\label{sec:CMS}
A direct application of the theory developed so far is the whole characterization of modulation spaces. Namely, the issue below generalizes Theorem 1.1 in \cite{CGshiftinvertible} to the quasi-Banach space setting, extending the indices $p,q\in [1,\infty]$ to $0<p,q\leq\infty$.  Whenever $p\not=q$ we need  the assumption $E_\cA$ upper-triangular, that is, the $2\times1$ block of $E_\cA$ in \eqref{defEAFA} satisfies $A_{21}=0_{d\times d}$. This requirement is needed for the use of Proposition \ref{thmA12}.
\begin{theorem} \label{thmF}
 Fix a non-zero window function $g\in \cS(\rd)$. Consider $0<p,q\leq\infty$, $W_\cA$  shift-invertible and a weight  $m\in\mathcal{M}_v(\rdd)$ with $m\asymp m\circ E_\cA^{-1}$.
 Then\\
 (i) For $0<p\leq\infty$   and we have
	\begin{equation}\label{charmod}
		f\in M^{p}_m(\rd) \qquad \Leftrightarrow \qquad W_\cA(f,g)\in L^{p}_m(\rdd),
	\end{equation}
	with equivalence of norms.\\
(ii) If we add the assumption that $E_\cA$ is upper-triangular, then 
\begin{equation}\label{charmod2}
	f\in M^{p,q}_m(\rd) \qquad \Leftrightarrow \qquad W_\cA(f,g)\in L^{p,q}_m(\rdd),
\end{equation}
with equivalence of norms.
\end{theorem}
\begin{proof}
 Take $f\in M^{p,q}_m(\rd)$. From the equality \eqref{WA-STFT} we infer 
	\begin{align*}
	|W_\cA(f,g)|(z)&=|\mathfrak{T}_{E_\cA^{-1}}\Phi_{M_\cA+L}V_{\widehat{\delta_\cA} g}f|(z)=|\mathfrak{T}_{E_\cA^{-1}}V_{\widehat{\delta_\cA} g}f|(z)\\
	&=|\det(E_{\cA})|^{-1/2}|V_{\widehat{\delta_\cA} g}f|(E_{\cA}^{-1}z).
\end{align*}
Since $\widehat{\delta_\cA}:\cS(\rd)\to \cS(\rd)$, we can choose the window ${\widehat{\delta_\cA} g}\in\cS(\rd)$ to compute the modulation space norm so that
$$\|	W_\cA(f,g)\|_{L^{p,q}_m}\asymp \|V_{\widehat{\delta_\cA} g}f(E_{\cA}^{-1}\cdot)\|_{L^{p,q}_m}.$$
The conclusion follows from Proposition \ref{thmA12}.
\end{proof}

In what follows we generalize {\cite[Corollary 3.12]{CGshiftinvertible}} to the quasi-Banach space setting $0<p,q\leq\infty$. 
\begin{theorem}\label{corWiener}
	Fix a non-zero window function $g\in \cS(\rd)$. Consider $0<p,q\leq\infty$, $W_\cA$  shift-invertible and  $m_1,m_2\in\cM_v(\rd)$ such that $m_2\asymp \cI m_2$, with $\cI m_2(x)=m_2(-x)$.
	Fix $g\in\cS(\rd)\setminus\{0\}$ and define
	\begin{equation}\label{tEA0}
		\tilde E_\cA=J E_\cA J,
	\end{equation}
with the symplectic matrix $J$ defined in \eqref{defJ}.
(Observe that $E_{\cA}^{-1}$ is lower triangular if and only if $\tilde E_\cA$ is upper triangular). If $m_1\otimes m_2\asymp (m_1\otimes m_2)\circ \tilde E_{\cA}^{-1}$ and $E_\cA$ is lower triangular, then
	\[
	\norm{f}_{W(\cF L^p_{m_1},L^q_{m_2})}\asymp\Big(\int_{\rd}\Big(\int_{\rd}|W_\cA(f,g)(x,\xi)|^pm_1(\xi)^pd\xi\Big)^{q/p}m_2(x)^qdx\Big)^{1/q},
	\]
	with the analogous for $\max\{p,q\}=\infty$.
\end{theorem}
\begin{proof}
As in the proof of Corollary 3.12 in \cite{CGshiftinvertible}, assuming $m_2(-x)=m_2(x)$, we can write
	\begin{equation*}
	\Big(\int_{\rd}\Big(\int_{\rd}|W_\cA(f,g)(x,\xi)|^pm_1(\xi)^pd\xi\Big)^{q/p}m_2(x)^qdx\Big)^{1/q}\asymp\norm{W_{\tilde\cA_0}(\hat f,\hat g)}_{L^{p,q}_{m_1\otimes m_2}},
\end{equation*}
	\[
\tilde \cA_0 = \begin{pmatrix}
	-A_{23} & A_{24} & A_{21} & -A_{22}\\
	A_{13} & -A_{14} & -A_{11} & A_{12}\\
	-A_{43} & A_{44} & A_{41} & -A_{42}\\
	A_{33} & -A_{34} & -A_{31} & A_{32}
\end{pmatrix},
\]
so that  $\tilde E_\cA = E_{\tilde \cA_0}$. The conclusion is due to Theorem \ref{thmF}
\end{proof}

If $p=q$  the additional assumption $ E_{\cA}^{-1}$ lower triangular is not needed. Observe that in this case $\norm{f}_{W(\cF L^p_{m_1},L^p_{m_2})}\asymp \|f\|_{M^p_{m_2\otimes m_1}}$, and the norm equivalence follows from Theorem \ref{thmF} above. In fact, notice that $$(m_1\otimes m_2)\circ E_\cA^{-1}\asymp (\cI m_2\otimes m_1)\otimes 	\tilde E_\cA^{-1}.$$

Consider a metaplectic Gabor frame  $\cG_\cA(g,\Lambda)$ and assume 
\begin{equation}\label{peso-m-equiv}
	m\asymp m\circ E_\cA^{-1},
\end{equation}  
then, for any $f\in M^{p,q}_m(\rd)$ we can use \eqref{piASI} to express  the  coefficient operator's entries $$|C_\cA f(\lambda)|=|\la f,\pi_\cA(\lambda)g\ra| =|\det(E_\cA)|^{-1/2}|\la f,\pi(E_\cA^{-1}\lambda)\widehat{\delta_\cA}g\ra|.$$

Observe that $\widehat{\delta_\cA}g\in \cS(\rd)$ for $g\in\cS(\rd)$, by Theorem \ref{thmFrames} $(ii)$; furthermore,     
$\cG(\widehat{\delta_\cA}g,E_\cA^{-1}\Lambda)$ is a Gabor frame with coefficient operator $C$ satisfying
$\|C f\|_{\ell^{p,q}_m(E^{-1}_\cA \Lambda)}\lesssim \|f\|_{\mpq_m}$, so that the equivalence of weights  in  \eqref{peso-m-equiv} gives
$$\|C_\cA f\|_{\ell^{p,q}_m(\Lambda)}=|\det(E_\cA)|^{-1/2}\|Cf\|_{\ell^{p,q}_m(E^{-1}_\cA \Lambda)}\lesssim \|f\|_{\mpq_m},$$
that is the boundedness of 
$C_\cA : M^{p,q}_m(\rd)\to \ell^{p,q}_m(\Lambda).$

Using the relation between $\pi_{\cA}(\lambda)$ and the \tfs\, $\pi(E_\cA^{-1}\lambda)$ displayed in \eqref{piASI}, and the equivalence of weights in \eqref{peso-m-equiv}, for any sequence $c_\lambda \in \ell^{p,q}_m(\Lambda)$, the sequence $\widetilde{c_\mu}:=c_ {E_\cA \mu}\Phi_{M_\cA+L}(\mu)\in \ell^{p,q}_m(E_\cA^{-1}\Lambda)$ so that
\begin{align*}\|D_\cA c_\lambda\|_{M^{p,q}_m(\rd)}&=\left\|\sum_{\lambda\in\Lambda} c_\lambda\pi_{\cA}(\lambda) g\right\|_{M^{p,q}_m(\rd)}  \asymp \left\|\sum_{\mu\in E_\cA^{-1}\Lambda}\widetilde{c_\mu}\pi(E_\cA^{-1}\lambda)\widehat{\delta_\cA} g\right\|_{M^{p,q}_m(\rd)}\\
&\lesssim \|(\widetilde{c_\mu})\|_{\ell^{p,q}_m(E_\cA^{-1}\Lambda)}\asymp  \|({c_\lambda})\|_{\ell^{p,q}_m(\Lambda)}.
\end{align*}

For the Banach space case $p,q\in [1,+\infty]$, the window class can be extended from $\cS(\rd)$ to $M^1_v(\rd)$. In fact, under the assumption \eqref{peso-m-equiv}, the metaplectic operator $\widehat{\delta_\cA}$ and its inverse  are bounded on $M^{1}_v(\rd)$, cf.  \cite[Theorem 4.6]{Fuhr}. Hence, $g\in M^{1}_v(\rd)\iff \widehat{\delta_\cA}g\in M^{1}_v(\rd)$ .
Arguing as for the Schwartz class and using the results for Gabor frames \cite[Chapter 12]{book} we infer that the coefficient operator $C_\cA$ is bounded from $M^{p,q}_m(\rd)$ to $\ell^{p,q}_m(\Lambda)$ and the other way round for the reconstruction operator $D_\cA$.

The observations above, together with the characterization of modulation spaces via Gabor frames (see, e.g., \cite[Theorem 3.2.37]{Elena-book} and \cite{Galperin2004}) yield an equivalent discrete  norm for modulation spaces in terms of metaplectic Gabor frames. Namely, 

\begin{theorem}\label{C6teomod}
	Consider  $\cG_\cA(g,\Lambda)$  a metaplectic Gabor frame 	for $\lrd$ with bounds $0<A\leq B$, with $g\in\cS(\rd)$ and canonical dual window $\gamma_\cA$ in \eqref{dualecanonica}. Assume $W_\cA$ shift-invertible and  $m\in\cM_v(\rdd)$, with $m\asymp m\circ E_{\cA}^{-1}$. Then,\\
	(i) For every $0<
	p,q\leq\infty$, $C_\cA:
	M^{p,q}_m(\rd)\to
	\ell^{p,q}_{{m}}(\Lambda)$ and
	$D_\cA:
	\ell^{p,q}_{{m}}(\Lambda)\to
	M^{p,q}_m(\rd)$ continuously. If $f\in\mpq_m(\rd),$ then
	the  expansions in \eqref{C3Gaborexp}
	converge	unconditionally in $\mpq_m$
	for $0< p,q<\infty$, and weak$^\ast$-${M^\infty_{1/v}}$ unconditionally if $p=\infty$ or $q=\infty$.\\
	(ii) The following (quasi-)norms are equivalent on $\mpq_m(\rd)$
	\begin{equation}\label{C3normequiv1}
		A\|f\|_{M^{p,q}_m(\rd)}\leq \|(\la f,\pi_\cA(\lambda)g\ra)_{\lambda\in\Lambda}\|_{\ell^{p,q}_m(\Lambda)}\leq B \|f\|_{M^{p,q}_m(\rd)},
	\end{equation}
	\begin{equation}\label{C3normequiv2} B^{-1}\|f\|_{M^{p,q}_m(\rd)} \leq   \|(\la f,\pi_\cA(\lambda)\gamma_\cA\ra)_{\lambda\in\Lambda}\|_{\ell^{p,q}_m(\Lambda)}\leq A^{-1} \|f\|_{M^{p,q}_m(\rd)}.
	\end{equation}
\end{theorem}
\begin{remark}
Assume $g,\gamma\in M^1_v(\rd)$ with $v$ satisfying \eqref{peso-m-equiv} and such that 
$$S_{\cA,g,\gamma}=D_{\cA,\gamma}C_{\cA,g}=I,\quad \mbox{on} \quad\lrd.$$
For $p,q\in[1,\infty]$, the statements of the previous theorem hold in the larger window class $M^1_v(\rd)$, with the canonical dual window $\gamma_\cA$ replaced by $\gamma$.
\end{remark}

\begin{appendix}
\section{}
	In \cite{CR2022}, the authors proved the following result, cf. \cite[Proposition 2.6]{CR2022}:
	\begin{proposition}\label{propA1}
		Let $\hat\cA\in Mp(2d,\bR)$ and $W_\cA$ be the corresponding metaplectic Wigner distribution. Then, there exists $\widehat{\cA_\ast}\in Mp(2d,\bR)$ such that for all $f,g\in L^2(\rd)$,
		\begin{equation}\label{A1}
			W_\cA(g,f)=\overline{W_{\cA_\ast}(f,g)}
		\end{equation}
		up to a sign.
	\end{proposition}
	In what follows we shall improve Proposition \ref{propA1}, carrying over the explicit expression of the projection $\cA_\ast$ in (\ref{A1}). First, we need to compute the intertwining relation between complex conjugation and metaplectic operators.
	
	\begin{proposition}\label{propA2}
		Let $\hat S\in Mp(d,\bR)$ be a metaplectic operator and $S=\pi^{Mp}(\hat S)$ have block decomposition (\ref{blocksA}). Define
		\[
			\bar S:=\begin{pmatrix}
			A & -B\\
			-C & D
			\end{pmatrix}.
		\] 
		Then, for all $f\in L^2(\rd)$,
		\[ 	
		\hat{S}\bar f = \overline{\hat{\bar S} f}.
		\]
	\end{proposition}
	
	\begin{proof}
		Let $T$ the operator defined by
		\[
			Tf = \overline{\hat S \bar f},\quad f\in L^2(\rd).
		\]
		Since $\hat S$ is a unitary operator on $L^2(\rd)$, $T$ is a unitary operator on $L^2(\rd)$. We have to prove that $T$ satisfies the intertwining relation in (\ref{muAdef}) for $\cA=\bar S$. For, let $z=(x,\xi)\in\rdd$ and take $\tau\in\bR$. Then, 
		\begin{align*}
			T\rho(z;\tau)f &= \overline{\hat S \overline{\rho(z;\tau)f}}\\
			&=\overline{\hat S \rho(x,-\xi;-\tau)\bar f}\\
			&=\overline{\rho(S(x,-\xi);-\tau)\hat S \bar f}\\
			&=\overline{e^{-2\pi i\tau}e^{-i\pi (Ax-B\xi)\cdot(Cx-D\xi)}\pi(Ax-B\xi,Cx-D\xi)\hat S \bar f}\\
			&=e^{2\pi i\tau}e^{-i\pi (Ax-B\xi)	\cdot(-Cx+D\xi)}\pi(Ax-B\xi,-Cx+D\xi)\overline{\hat S \bar f}\\
			&=\rho(\bar S(x,\xi);\tau)Tf,
		\end{align*}
	as desired.
	\end{proof}
	
	\begin{corollary}\label{corA}
		Under the assumptions of Proposition \ref{propA1}, we have
		$$\cA_\ast=\overline{\cA} \cD_L$$ with the matrix $L$ defined in \eqref{defL}. Namely, if $\cA$ has block decomposition (\ref{blockDecA}), $\cA_\ast$ is given by
		\begin{equation}\label{matrixAstar}
		\cA_\ast=\begin{pmatrix}
			A_{12} & A_{11} & -A_{14} & -A_{13}\\
			A_{22} & A_{21} & -A_{24} & -A_{23}\\
			-A_{32} & -A_{31} & A_{34} & A_{33}\\
			-A_{42} & -A_{41} & A_{44} & A_{43}
		\end{pmatrix}.
		\end{equation}
	\end{corollary}

	\begin{proof}
		 Observe that $\widehat{ \cD_L} F(x,y)=F(y,x)$, so that, for every $f,g\in L^2(\rd)$,
		\[
			g\otimes f(x,y)=f(y)g(x)=f\otimes g(y,x)=f\otimes g(\cD_L (x,y))=\widehat{\cD_L}(f\otimes g)(x,y).
		\]
		By Proposition \ref{propA2}, it follows that, up to a sign,
		\begin{align*}
			W_\cA(f,g)=\hat \cA (f\otimes\bar g)=\hat \cA(\overline{\bar f\otimes g})=\overline{\widehat{\overline{\cA}}(\bar f\otimes g)}=\overline{\widehat{\overline{\cA}\cD_L}(g\otimes \bar f)}=\overline{W_{\cA_\ast}(g,f)}.
		\end{align*}
	Assuming that $\cA$ exhibits the block decomposition (\ref{blockDecA}), a straightforward computation yields \eqref{matrixAstar}.
	This concludes the proof.
	\end{proof}
	\begin{remark}
		A straightforward computation shows that $\overline{S}^T=\overline{S^T}$. In fact, if $S$ has block decomposition (\ref{blocksA}),
		\begin{align*}
		\overline{S}^T=\begin{pmatrix}
		A^T & -C^T\\ -B^T & D^T
		\end{pmatrix},
		\end{align*}
		whereas
		\[
			{S^T}=\begin{pmatrix}A^T & C^T\\ B^T & D^T\end{pmatrix}, \quad so \ that \quad \overline{S^T}=\begin{pmatrix}
		A^T & -C^T\\ -B^T & D^T
		\end{pmatrix}=\overline{S}^T.
		\]
	\end{remark}
\end{appendix}

\section*{Acknowledgements}
The authors have been supported by the Gruppo Nazionale per l’Analisi Matematica, la Probabilità e le loro Applicazioni (GNAMPA) of the Istituto Nazionale di Alta Matematica (INdAM). The second author was supported by the University of Bologna and HES-SO Valais - Wallis School of Engineering. He also acknowledge the support of The Sense Innovation and Research Center, a joint venture of the University of Lausanne (UNIL), The Lausanne University Hospital (CHUV), and The University of Applied Sciences of Western Switzerland – Valais/Wallis (HES-SO Valais/Wallis).

\end{document}